\documentclass[11pt]{amsart}
\usepackage{amsmath,mathtools}
\usepackage{amsmath, amsthm, amssymb, amsfonts, enumerate}
\usepackage[colorlinks=true,linkcolor=blue,urlcolor=blue]{hyperref}
\usepackage{dsfont}
\usepackage{color}
\usepackage{geometry}
\usepackage{todonotes}
\usepackage{epstopdf}
\usepackage{bbm}
\usepackage{geometry}
\usepackage[utf8]{inputenc}
\usepackage{bm}
\usepackage{amsfonts}
\usepackage{amsfonts}
\usepackage{textcomp}
\usepackage{amssymb}
\usepackage{float}
\usepackage{tikz}
\usepackage{epsfig}
\usepackage{amsmath}
\usepackage[english]{babel}
\usepackage{a4}
\usepackage{enumerate}
\usepackage{amsmath}
\newcommand{\stkout}[1]{\ifmmode\text{\sout{\ensuremath{#1}}}\else\sout{#1}\fi}
\newtheorem{theorem}{Theorem}[section]
\newtheorem{remark}[theorem]{Remark}
\newtheorem{assumption}[theorem]{Assumption}
\newtheorem{lemma}[theorem]{Lemma}
\newtheorem{proposition}[theorem]{Proposition}

\def \E{\mathsf{E}}

\def \P{\mathsf{P}}
\def \R{\mathbb{R}}

\def\d{\mathrm{d}}

\DeclareMathOperator*{\argmin}{argmin}
\DeclareMathOperator*{\Span}{Span}

\definecolor{red}{rgb}{1.0,0.0,0.0}

\definecolor{blu}{rgb}{0.0,0.0,1.0}

\definecolor{gre}{rgb}{0.03,0.50,0.03}

\title[Epidemic Control via Lockdown]{Taming the Spread of an Epidemic by Lockdown Policies}
\author[Federico]{Salvatore Federico}
\author[Ferrari]{Giorgio Ferrari}
\address{S.~Federico: Dipartimento di Economia, Universit\`a di Genova, Via F.\ Vivaldi 5, 16126, Genova, Italy}
\email{\href{mailto:salvatore.federico@unige.it}{salvatore.federico@unige.it}}
\address{G.~Ferrari: Center for Mathematical Economics (IMW), Bielefeld University, Universit\"atsstrasse 25, 33615, Bielefeld, Germany}
\email{\href{mailto:giorgio.ferrari@uni-bielefeld.de}{giorgio.ferrari@uni-bielefeld.de}}
\date{\today}

\numberwithin{equation}{section}

\begin{document}

\begin{abstract} 
We study the problem of a policymaker who aims at taming the spread of an epidemic while minimizing its associated social costs. The main feature of our model lies in the fact that the disease's transmission rate is a diffusive stochastic process whose trend can be adjusted via costly confinement policies. We provide a complete theoretical analysis, as well as numerical experiments illustrating the structure of the optimal lockdown policy. In all our experiments the latter is characterized by three distinct periods: the epidemic is first let freely evolve, then vigorously tamed, and finally a less stringent containment should be adopted. Moreover, the optimal containment policy is such that the product ``reproduction number $\times$ percentage of susceptible'' is kept after a certain date strictly below the critical level of one, although the reproduction number is let oscillate above one in the last more relaxed phase of lockdown. {{Finally, an increase in the fluctuations of the transmission rate is shown to give rise to an earlier beginning of the optimal lockdown policy, which is also diluted over a longer period of time.}} 
\end{abstract}

\maketitle

\smallskip


\smallskip

{\textbf{MSC2010 subject classification}}: 93E20, 49N90, 92D30, 97M40.
  
\smallskip


%%%%%%%%%%%%%%%%%%%%%%%%%%%%%%%%%%%%%%%%%

\section{Introduction}
\label{introduction}

During the current Covid-19 pandemic, policymakers are dealing with the trade-off between safeguarding public health and damming the negative economic impact of severe lockdowns. The fight against the virus is made especially hard by the absence of a vaccination and the consequent random horizon of any policy, as well as by the extraordinariness of the event. In particular, the lack of data from the past, the difficulty of rapidly and accurately tracking infected, and super-spreading events such as mass gatherings, give rise to a random behavior of the transmission rate/reproduction number of the virus (see, e.g., \cite{Hotz}\footnote{Refer also to the website \url{https://stochastik-tu-ilmenau.github.io/COVID-19/index.html}}). In this paper we propose and study a model for the optimal containment of infections due to an epidemic in which both the time horizon and the transmission rate of the disease are stochastic.

In the last months, the scientific literature experienced an explosion in the number of works where the statistical analysis and the mathematical modeling of epidemic models is considered, as well as the economic and social impact of lockdown policies is investigated. 
%Clearly, providing a careful review of such a huge number of contributions is outside the scope of this paper, and we will then just briefly discuss those works that we believe share more similarities with ours.
A large bunch of papers provides numerical studies related to the Covid-19 epidemics in the setting of classical epidemic models or of generalization of them. Among many others, we refer to \cite{Alvarez}, that studies numerically optimal containment policies in the context of a \emph{Susceptible-Infected-Recovered} (SIR) model (cf.\ \cite{KermackMcKendrick}); \cite{Kantner} which also allows for seasonal effects; \cite{Toda}, which estimates the transmission rate in various countries for a SIR model with given and fixed transmission rate; \cite{Asprietal}, which combines a careful numerical study with an elegant theoretical study of optimal lockdown policies in the SEAIRD (susceptible (S), exposed (E), asymptomatic (A), infected (I), recovered (R), deceased (D)) model; \cite{Erhan}, where a detailed numerical analysis is developed for a SIR model of the Covid-19 pandemic in which herd immunity, behavior-dependent transmission rates, remote workers, and indirect externalities of lockdown are explicitly considered; \cite{Acemoglu}, where -- in the context of a multi-group SIR model -- it is investigated the effect of lockdown policies which are targeted to different social groups (especially, the ``young'', the ``middle-aged'' and the ``old''); \cite{Gollier}, in which a multi-risk SIR model with heterogeneous citizens is calibrated on the Covid-19 pandemic in order to study the impact on incomes and mortality of age-specific confinements and Polymerase chain reaction (PCR) tests; \cite{Favero}, which calibrates and tests a SEIRD model (susceptible (S), exposed (E), infected (I), recovered (R), deceased (D)) of the spread of Covid-19 in an heterogeneous economy where different age and sectors are related to distinct risks.

A theoretical study of the optimal confinement policies in epidemic models is usually challenging because of the nonlinear structure of the underlying dynamical system. The first results on a control-theoretic approach to confinement policies are perhaps those presented in Chapter 4 of \cite{Behncke}, where it is shown that the optimal policy depends only on the shadow price difference between infected and susceptible. In the context of an optimal timing problem, \cite{Jacco} uses a continuous-time Markov chain model to study the value and optimal exercise decision of two (sequential) options: the option to intervene on the epidemic and, after intervention has started, the option to end the containment policies. Control-theoretic analysis are also presented in the recent \cite{KruseStrack} and \cite{Micloetal}. In \cite{Micloetal} the authors study a deterministic SIR model in which the social planner acts in order to keep the transmission rate below its natural level with the ultimate aim not to overwhelm the national health-care system. 
%As in our case, in \cite{Micloetal} it is shown that the optimal control is characterized by three distinct phases where lockdown should be absent, severe, and finally mild. 
The minimization of a social cost functional is instead considered in \cite{KruseStrack}, in the context of a deterministic SIR model over a finite time-horizon. The resulting control problem is tackled via the Pontryagin maximum principle and then a thorough numerical illustration is also provided. 

Inspired by the deterministic problems of \cite{KruseStrack} and \cite{Micloetal} (see also \cite{Acemoglu, Alvarez}, among others), and motivated by the need of incorporating random fluctuations in the disease's transmission rate, in this paper we consider a stochastic control-theoretic version of the classical SIR model of Kermack and McKendrick \cite{KermackMcKendrick}. A population with finite size is divided in three different groups: healthy people that are susceptible to the disease, infected individuals, and people that have recovered (and are not anymore susceptible) or dead. However, differently to the classical SIR model, we suppose that disease's transmission rate is time-dependent and stochastic. In particular, it evolves as a general diffusion process whose trend can be adjusted by a social planner through policies like social restrictions and lockdowns. The randomness in the transmission rate is modeled by a Wiener process representing all those factors affecting the transmission rate and that are not under the direct control of the regulator. The social planner faces the trade-off between the expected social and economic costs (e.g., drops in the gross domestic product) arising from severe restrictions and the expected costs induced by the number of infections that -- if uncontrolled -- might strongly impact on the national health-care system and, more in general, on the social well-being. The social planner aims at minimizing those total expected costs up to the time at which a vaccination against the disease is discovered. In our model, such a time is also random and independent of the Wiener process.

We provide a complete theoretical study of our model by showing that the minimal cost function (value function) is a classical twice-continuously differentiable solution to its corresponding Hamilton-Jacobi-Bellman (HJB) equation, and by identifying an optimal control in feedback form\footnote{The aforementioned regularity of the value function is a remarkable result on its own. Indeed, although the state process is degenerate (as the Wiener process only affects the dynamics of the transmission rate), we can show that the so-called H\"ormander's condition (cf.\ \cite{Nualart}) holds true for any choice of the model's parameters. This then ensures the existence of a smooth probability transition density for the underlying (uncontrolled) stochastic process, which in turn enables to prove substantial regularity of the value function.}. 
From a technical point of view, the main difference between the models in \cite{Acemoglu, Alvarez, Erhan, KruseStrack, Micloetal} and ours, is that we deal with a stochastic version of the SIR model, instead of a deterministic one. As a matter of fact, in the aforementioned works the transmission rate is a deterministic control variable, while it is a controlled stochastic state variable in our paper. Moreover, our formulation is also different from that of other stochastic SIR models where the random transmission rate is chosen in such a way that only the levels of infected and susceptible people become affected by noise, with the transmission rate itself not being a state variable (see, e.g., \cite{Jiang,Tornatore} and references therein).
To the best of our knowledge, ours is the first work considering the transmission rate as a diffusive stochastic state variable and providing the complete theoretical analysis of the resulting control problem.

{{In addition to its theoretical value,}} the determination of an optimal control in feedback form allows us to perform numerical experiments aiming at showing some implications of our model. For the numerical analysis we specialize the dynamics of the transmission rate, that we take to be mean-reverting and bounded between $0$ and some $\gamma>0$ (cf.\ \eqref{eq:ZOUbis}). In this case study, the containment policies employed by the social planner have the effect of modifying the long-run mean of the transmission rate, towards which the process converges at an exponential rate. Moreover, we take a separable social cost function (cf.\ \eqref{costexample}). This is quadratic both in the regulator's effort and in the percentage of infected people. %The model's parameters are then selected in order to be in line with those of the current Covid-19 epidemic. 
%In particular we assume that the average length of an infection equals 18 days, the level of the maximal possible transmission rate of the disease is $\gamma= 0.16$, and the natural transmission rate of the disease is $0.1$. Towards this value the transmission rate reverts at rate $\vartheta=0.1$ in absence of any interventions of the social planner. Moreover, the volatility coefficient of the transmission rate is $\sigma=0.1$. 

An interesting effect which is in fact common to all our numerical experiments is that the optimal lockdown policy is characterized by three distinct periods. In the first phase it is optimal to let the epidemic freely evolve, then the social restrictions should be stringent, and finally should be gradually relaxed in a third period. We also investigate which is the effect of the maximal level $L$ of allowed containment measures (i.e., the lockdown policy can take values in $[0,L]$) on the final percentage of recovered, which in fact turns out to be decreasing with respect to $L$. This then suggests that the case $L=1$ -- which leads in a shorter period to the definitive containment of the disease with the smallest percentage of final recovered --  might be thought of as optimal in the trade-off between social costs and final number of recovered.

We observe that if the epidemic spread is left uncontrolled, then its reproduction number $(\mathcal{R}_t)_t$ fluctuates around $1.8$ and the final percentage of recovered (i.e.\ the total percentage of infected during the disease) is approximately $72\%$ of the society after circa 7 months (in all our simulations the initial infected were $1\%$ of the population). On the other hand, when $L=1$, under the optimal policy we have a relative reduction of circa $30\%$ of the total percentage of recovered individuals, and the reproduction number drops below $0.6$ in the period of severe lockdown (circa 60 days). Moreover, the optimal containment is such that the so-called ``herd immunity'' is reached as the product $\mathcal{R}_tS_t$ (reproduction number $\times$ percentage of susceptible) becomes strictly smaller than the critical level of one, even if $\mathcal{R}_t$ oscillates at around $1.7$ in the last more relaxed phase of lockdown. {{Finally, we observe that an increase of the fluctuations of the transmission rate $\beta$ have the effect of anticipating the beginning of the lockdown policies, of diluting the actions over a longer period, and of keeping a larger level of containment in the long run. This can be explained by thinking that an higher uncertainty in the transmission rate induces the policymaker to act earlier and over a longer period in order to prevent positive larger shocks of $\beta$.}}

The rest of the paper is organized as follows. In Section \ref{sec:SIR} we set up the model and the social planner problem. In Section \ref{sec:mainresult} we develop the control-theoretic analysis and provide the regularity of the minimal cost function and an optimal control in feedback form. In Section \ref{sec:numerics} we present our numerical examples, while concluding remarks are made in Section \ref{sec:concl}. Finally, Appendix \ref{sec:app} collects the proof of some technical results needed in Section \ref{sec:mainresult}.

%%%%%%%%%%%%%%%%%%%%%%%%%%%%%%%%%%%

\section{Problem Formulation}
\label{sec:setting}

\subsection{The Stochastic Controlled SIR Model}
\label{sec:SIR}

We model the spread of the infection by relying on a generalization of the classical SIR model that dates back to the work by Kermack and McKendrick \cite{KermackMcKendrick}. The society has population $N$ and it consists of three different groups. The first group is formed by those people who are healthy, but susceptible to the disease; the second group contains those who are infected, while the last cohort consists of those who are recovered or dead. In line with the classical SIR model, we assume that, once recovered, an individual stays healthy for ever. We denote by $S_t$ the percentage of individuals who are susceptible at time $t\geq0$, by $I_t$ the percentage of infected, and by $R_t$ the fraction of recovered or dead. Clearly, $S_t + I_t + R_t=1$ for all $t\geq0$.

The fraction of infected people grows at a rate which is proportional to the fraction of society that it is still susceptible to the disease. In particular, letting $\beta_t$ be the instantaneous transmission rate of the disease, during an infinitesimal interval of time $\d t$, each infected individual generates $\beta_t S_t$ new infected individuals. It thus follows that the percentage of healthy individuals that get infected within $\d t$ units of time is
$I_t  \beta_t S_t.$

Notice that the instantaneous transmission rate $\beta_t$ measures the disease's rate of infection, as well as the the average number of contacts per person per time. In this regard, $\beta_t$ can be thus influenced by a social planner via policies that effectively cap the social interaction, like social distancing and lockdown.

During an infinitesimal interval of time $\d t$, the fraction of infected is reduced by $\alpha I_t$, since infected either recover from the disease, or die because of it at a rate $\alpha>0$. 

According to the previous considerations, the dynamics of $S_t$ and $I_t$ can be thus written as
\begin{equation}
\label{eq:S}
\d S_t = - \beta_t S_t I_t \d t, \quad t>0, \qquad S_0= x,
\end{equation}
and
\begin{equation}
\label{eq:I}
\d I_t = \big(\beta_t S_t I_t - \alpha I_t\big) \d t, \quad t>0, \qquad I_0= y,
\end{equation}
where $(x,y) \in (0,1)^2$ are given initial values such that\footnote{The choice of considering $x+y<1$ -- i.e.\ of having an initial strictly positive percentage of recovered -- is only done in order to deal with an open set in the subsequent mathematical formulation of the problem. As a matter of fact, such a condition is not restrictive from the technical point of view as our results still apply if $x+y<\ell$, for some $\ell>1$, thus covering the case $x+y=1$ as well.} $x+y \in (0,1)$.

Notice that for any $t\geq0$, and for any choice of $(\beta_t)_t$ we can write
\begin{equation}
\label{eq:SolSI}
S_t = x e^{-\int_0^t \beta_u I_u \d u} \qquad \text{and} \qquad I_t = y e^{-\alpha t + \int_0^t \beta_u S_u \d u},
\end{equation}
and therefore $S_t>0$ and $I_t>0$ for all $t\geq0$. Moreover, summing up \eqref{eq:S} and \eqref{eq:I} we have $\d(S_t + I_t) = -\alpha I_t <0$ for all $t>0$, which then implies that $S_t + I_t < 1$ for all $t\geq0$. 

We depart from the classical SIR model by assuming that the transmission rate $\beta_t$ is time-varying, stochastic, and may be controlled. More precisely, we let $(\Omega, \mathcal{F}, \mathbb{F}:=(\mathcal{F}_t)_t, \P)$ be a complete filtered probability space with filtration $\mathbb{F}$ satisfying the usual conditions, and we define on that a one-dimensional Brownian motion $(W_t)_t$.
For a given and fixed $L\geq 0$, and for any $(\xi_t)_t$ belonging to
$$\mathcal{A}:=\big\{\xi:\, \Omega \times [0,\infty) \to [0,L],\,\, (\xi_t)_t\,\, \mathbb{F}-\text{progressively measurable}\big\},$$
we assume that the transmission rate evolves according to the stochastic differential equation
\begin{equation}
\label{eq:Z}
\d \beta_t = b(\beta_t,\xi_t) \d t + \sigma(\beta_t)\d W_t, \quad t>0, \qquad \beta_0= z >0.
\end{equation}
The process $(\xi_t)_t$ influences the trend of the transmission rate and it should be interpreted as any effort devoted by the social planner to the decrease of the transmission rate. In this sense, $\xi=0$  corresponds to the case of no effort done to decrease the disease, whereas the case $\xi=L$ corresponds to the maximal effort. To fix ideas, $\xi_t$ may represent a percentage of social/working lockdown at time $t$ and $L$ corresponds to the maximal implementable value of such lockdown (e.g. $60\%$, etc.). On the other hand, the Brownian motion $(W_t)_t$ models any shock affecting the transmission rate and which is not under the control of the social planner. 

Regarding the dynamics of $(\beta_t)_t$ we make the following \textbf{standing assumption}.
{{
\begin{assumption}
\label{assZ}
\hspace{10cm}
\begin{itemize}
\item[(i)] For every $\xi\in\mathcal{A}$, there exists a unique strong solution to \eqref{eq:Z} and it lies in an open interval $\mathcal{I}\subseteq(0,\infty)$.

\item[(ii)] $b:\mathcal{I}\times [0,L] \to \R$ is bounded, infinitely many times continuously differentiable with respect to its first argument, and has bounded derivatives of any order; that is, there exists $K_b>0$ such that
$$\sup_{n \in \mathbb{N}} \sup_{(z,\xi) \in \mathcal{I}\times [0,L]}\big|\frac{\partial^n}{\partial z^n}b(z,\xi)\big| \leq K_b.$$

\item[(iii)] $\sigma:\mathcal{I} \to (0,\infty)$ is bounded, infinitely many times continuously differentiable with respect to its first argument, and has bounded derivatives of any order; that is, there exists $K_{\sigma}>0$ such that
$$\sup_{n \in \mathbb{N}} \sup_{z \in \mathcal{I}}|\sigma^{(n)}(z)| \leq K_{\sigma}.$$

%\item[(iv)] $\sigma^2>0$ {on} $\mathcal{I}$.

%\item[(iv)] $\xi \mapsto b(z,\xi)$ is nonincreasing for any $z \in \mathcal{I}$.
\end{itemize}
\end{assumption}
}}

A reasonable dynamics of the transmission rate $(\beta_t)_t$ is the mean-reverting 
\begin{equation}
\label{eq:ZOU}
\d \beta_t = \vartheta\Big(\widehat{\beta}\big(L-\xi_t\big) - \beta_t\Big) \d t + \sigma\beta_t(\gamma - \beta_t)\d W_t, \quad t>0, \qquad \beta_0= z \in (0,\gamma),
\end{equation}
for some $\vartheta, \gamma, \sigma>0$,  $\widehat{\beta} \in (0,\gamma)$. In this case, in can be shown that $0$ and $\gamma$ are unattainable by the diffusion $(\beta_t)_t$, which then takes values in the interval $\mathcal{I}=(0,\gamma)$ for any $t\geq 0$. The level $\widehat{\beta}$ can be seen as the natural transmission rate of the disease, towards which the transmission rate reverts at rate $\vartheta$ when $\xi\equiv0$. Finally, the level $\gamma$ is the maximal possible transmission rate of the disease, and $\sigma$ is a measure of the fluctuations of $(\beta_t)_t$ around $\widehat{\beta}$. We will employ this dynamics in our numerical illustrations (cf.\ Section \ref{sec:numerics} below). {{Notice that dynamics \eqref{eq:ZOU} fulfills all the requirements of Assumption \ref{assZ}; this is shown, for the sake of completeness, in Proposition \ref{prop:betadyn} in Appendix \ref{sec:app}. Moreover, if $\xi_t \equiv L$, then the transmission rate defined through \eqref{eq:ZOU} reaches $0$ asymptotically, as its drift is negative and its diffusion coefficient stays bounded. Hence, under the maximal lockdown policy, the disease is asymptotically eradicated.}}

{{
\begin{remark}
\label{rem:micromodel}
Uncertainty comes into our model through the transmission rate. A fully probabilistic setting, in which the transitions from $S$ to $I$ to $R$ are driven by Poisson processes, can be formulated by following the discussion of Example 2 in \cite{Greenwood-Gordillo} or Section 3 in \cite{Allen2017} (see also the recent \cite{Donsimoni}). We describe informally such a model in the sequel.

Assume that the fraction of susceptible and infected $(S_t,I_t)_t$ follow a continuous-time Markov chain with state space $\mathcal{S}:=\{(s,i) \in \{0,\frac{1}{N},\dots,1\}^2: i + s < 1\}$, where $N$ denotes the fixed population's size. For a small time interval $\Delta t$ let
$$p_{(s,i), (s+k,i+j)}(\Delta t) := \P\big( (S_{t+\Delta t}, I_{t+\Delta t}) = (s+k,i+j) \big| (S_{t}, I_{t}) = (s,i)\big),$$ 
and assume that
\begin{equation*}
p_{(s,i), (s+k,i+j)}(\Delta t)=\left\{
\begin{array}{lr}
\displaystyle \beta_t I_t S_t \Delta t + o(\Delta t)\,\,\qquad \qquad \quad \,\,\,\,\, \mbox{if $(k,j)=(-\frac{1}{N},\frac{1}{N})$}\\[+14pt]
\displaystyle \alpha I_t \Delta t + o(\Delta t) \,\,\qquad \qquad \qquad \,\,\,\,\,\,\, \mbox{if $(k,j)=(0,-\frac{1}{N})$}\\[+14pt]
\displaystyle \big(1 - \alpha I_t  - \beta_t I_t S_t\big)\Delta t + o(\Delta t) \,\,\,\,\,\, \mbox{if $(k,j)=(0,0)$}\\[+14pt]
\displaystyle o(\Delta t) \,\,\qquad \qquad \qquad \qquad \qquad\,\,\,\,\,\,\, \mbox{otherwise}.
\end{array}
\right.
\end{equation*}
It thus follows that the increments $\Delta S_t := S_{t+\Delta t} - S_{t}$ and $\Delta I_t := I_{t+\Delta t} - I_{t}$ can be written as
\begin{equation}
\label{eq:DeltaS}
\Delta S_t = - \beta_t I_t S_t \Delta t  + \Delta N^{(1)}_t
\end{equation}
and
\begin{equation}
\label{eq:DeltaI}
\Delta I_t =  \Big(\beta_t I_t S_t - \alpha I_t\Big) \Delta t  - \Delta N^{(1)}_t + \Delta N^{(2)}_t,
\end{equation}
where $\Delta N^{(1)}_t$ and $\Delta N^{(2)}_t$ are conditionally centered Poisson increments with zero mean and conditional variances $\beta_t I_t S_t \Delta t$ and $\alpha I_t \Delta t$, respectively.

On the other hand, let $\Delta t$ be such that $\Delta t = \sum_{i=1}^n \Delta t_i$, where $\Delta t_i = t_i - t_{i-1}$, $i=1,\dots,n$, $t_0=t$ and $t_n=t + \Delta t$.
Then, the increment of the transmission rate (in absence of any social planner's intervention) can be written as
$$\Delta \beta_t = \sum_{i=1}^n \Delta \beta_{t_i},$$
where $\Delta \beta_t:=\beta_{t + \Delta t} - \beta_t$ and $\Delta \beta_{t_i}= \beta_{t_i} - \beta_{t_{i-1}}$. If $\Delta t_i$ are sufficiently small, we can reasonably argue that the random variables $\Delta \beta_{t_i}$ on the interval $\Delta t$ are independent and identically distributed. For $n$ sufficiently large, the Central Limit Theorem implies that $\Delta \beta_t$ has an approximate Gaussian distribution. We assume that such a mean is $b(\beta_t,0) \Delta t$ and the variance $\sigma^2(\beta_t) \Delta t$, so that
\begin{equation}
\label{eq:Deltabeta}
\Delta \beta_t =   b(\beta_t,0) \Delta t + \sigma(\beta_t) \Delta W_t,
\end{equation}
for a standard Gaussian increment $\Delta W_t$.

For $\Delta t$ sufficiently small, the previous \eqref{eq:DeltaS}, \eqref{eq:DeltaI} and \eqref{eq:Deltabeta} define a fully probabilistic model in which uncertainty in $S$ and $I$ is driven by jump processes and, indirectly, through the diffusive transmission rate $\beta$. Then, allowing for a governmental control of $\beta$ as in \eqref{eq:ZOU} would result into an interesting stochastic control problem with jump-diffusive dynamics that we leave for future research.
\end{remark}

\begin{remark}
\label{rem:impulse}
Another modeling feature that needs some discussion regards the nature of the control rule in \eqref{eq:ZOU}. In our formulation, the policymaker adjusts $\xi$ continuously over time with the aim of decreasing the trend of the transmission rate. However, motivated by the real-world strategies employed during the Covid-19 crisis, one can very well imagine a model where regulatory constraints are introduced once the reproduction number $\mathcal{R}_t =\beta_t/\alpha$ becomes larger than a certain value, say $\mathcal{R}^{\star}$. Within this setting, a natural question would be: which is the optimal $\mathcal{R}^{\star}$ and the optimal size of interventions? A possible answer to this question could be found by proposing a model where the policymaker instantaneously reduces the level of $\beta$ via lockdown policies and faces proportional and/or fixed costs for its actions. This would gives rise to a singular or impulsive stochastic control problem; see \cite{AlvarezL}, \cite{Ferrari} and \cite{Belak} and references therein. Given the underlying multi-dimensional setting, we expect that the optimal trigger level $\mathcal{R}^{\star}$ would be a function of the current values of $(S_t,I_t)$. However, the proof of such a conjecture would require the thorough study of a complex (non convex) three-dimensional degenerate singular/impulse stochastic control problem that clearly requires techniques different from those employed in this work.
\end{remark}
}}

%%%%%%%%%%%%%%%%%%%%%%%%%%%%

\subsection{The Social Planner Problem}
\label{sec:PB}

The epidemic generates social costs, that we assume to be increasing with respect to the fraction of the population that is infected. These costs might arise because of lost gross domestic product (GDP) due to inability of working, because of an overstress of the national health-care system etc. The social planner thus employs policies $(\xi_t)_t$ in the form, e.g., of social distancing or lockdown in order to adjust the growth rate of the transmission rate $\beta$, with the aim of effectively flattening the curve of the infected percentage of the society. Such actions however come with a cost, which increases with the amplitude of the effort. 
Assuming that a vaccination against the disease is discovered at a random time $\tau$ exponentially distributed with parameter $\lambda_o>0$ and independent of $(W_t)_t$\footnote{We are implicitly requiring that the underlying probability space $(\Omega, \mathcal{F}, \mathbb{F}:=(\mathcal{F}_t)_t, \P)$ is rich enough to accommodate also such an exponential time $\tau$.} (see also Remark \ref{rem:tau} below), the social planner aims at solving
\begin{equation}
\label{eq:Probl0}
\inf_{\xi \in \mathcal{A}}\E\bigg[\int_0^{\tau} {e^{-\delta t}} C\big(I_t, \xi_t\big) \d t\bigg].
\end{equation}
Here, $\delta\geq 0$ measures the social planner's time preferences, and $C:[0,1]\times [0,L] \to [0,\infty)$ is a running cost function measuring the negative impact of the disease on the public health as well as the economic/social costs induced by lockdown policies. The following requirements are satisfied by $C$.
\begin{assumption}
\label{assC}
\hspace{10cm}
\begin{itemize}
\item[(i)] $(y,\xi) \mapsto C(y,\xi)$ is convex and  continuous on $[0,1]\times[0,L]$.  
\item[(ii)] For any $y\in [0,1]$ we have that $\xi \mapsto C(y,\xi)$ is nondecreasing.
\item[(iii)] For any $\xi \in [0,L]$ we have that $y \mapsto C(y,\xi)$ is nondecreasing.
\item[(iv)] There exists $K>0$ such that for any $\xi\in[0,L]$ we have that
$$|C(y,\xi) - C(y',\xi)| \leq K|y-y'|, \quad \forall (y,y')\in [0,1]^2.$$
\item[(v)] $y\mapsto C(y,\xi)$ is semiconcave\footnote{{{A function $f:\mathbb{R}^n \to \R$, $n\geq 1$, is called \emph{semiconvex} if there exists a constant $K\geq 0$ such that $f(x) + \frac{K}{2}|x|^2$ is convex; it is \emph{semiconcave} if $-f$ is semiconvex.} }} on $[0,1]$, uniformly with respect to $\xi\in[0,L]$; that is, there exists $K>0$ such that for any $\xi\in[0,L]$ and any $\mu \in [0,1]$ one has
$$\mu C(y,\xi) + (1-\mu) C(y',\xi) - C\big(\mu y + (1-\mu) y',\xi\big) \leq K\mu(1-\mu)|y-y'|^2, \quad \forall (y,y')\in [0,1]^2.$$
\end{itemize}
\end{assumption}
Without loss of generality, we also take $C(0,0)=0$. Convexity of $y\mapsto C(y,\xi)$ captures the fact that the social costs from the disease might be higher if a large share of the population is infected since, for example, the social health-care system is overwhelmed. The fact that $\xi\mapsto C(y,\xi)$ is convex describes that marginal costs of actions are increasing because, e.g., an additional lockdown policy might have a larger impact on an already stressed society. Finally, the Lipschitz and semiconcavity property of $C(\cdot,\xi)$ are technical requirements that will be important in the next section.

An application of Fubini's theorem, employing the independence of $\tau$ and $(W_t)_t$, allows to rewrite the problem defined in \eqref{eq:Probl0} as
\begin{equation}
\label{eq:Probl}
\inf_{\xi \in \mathcal{A}}\E\bigg[\int_0^{\tau} {e^{-\delta t}} C\big(I_t, \xi_t\big) \d t\bigg] = \inf_{\xi \in \mathcal{A}}\E\bigg[\int_0^{\infty} e^{-\lambda t} C\big(I_t, \xi_t\big) \d t\bigg],
\end{equation}
where $\lambda:=\lambda_o+\delta$.

{{
\begin{remark}
\label{rem:tau}
The assumption that a vaccination against the disease is discovered at an exponential random time $\tau$, independent of $(W_t)_t$, has the technical important effect of leading to a time-homogeneous social planner problem (cf.\ \eqref{eq:Probl} above). From a modeling point of view, such a requirement is clearly debatable, as it presupposes that the decision maker does not take into account the scientific progress in the epidemic's treatment. In order to take care of this, we now propose an alternative more realistic formulation which, however, comes at the cost of substantially increasing the mathematical complexity of the social planner problem.

Suppose that the social planner has full information about the current technological level $Q_t$ achieved in the disease's treatment and assume, for example, that this evolves according to the SDE:
$$\d Q_t = \mu(Q_t) \d t + \eta(Q_t) \d B_t, \quad Q_0=q \in \R_+,$$
for suitable $\mu$ and $\eta$, and for a standard Brownian motion $(B_t)_t$ independent of $(W_t)_t$. The process $(B_t)_t$ models all the exogenous shocks affecting the technological achievements (e.g., new scientific discoveries in related fields), while $\mu$ measures the instantaneous trend of the research. Define then a continuous-time Markov chain $(M_t)_t$ with two states, $0$ and $1$, where $0$ means that the vaccination is not available and $1$ that a treatment has been instead found. We assume that $1$ is an absorbing state and that the Markov chain has transition rate from state $0$ to state $1$ given by $(\lambda(t, Q_t))_t$. Here, $\lambda: \R_+ \times \R_+ \mapsto \R_+$ is such that $\Lambda_t:= \int_0^t \lambda(s,Q_s) \d s < \infty$, a.s.\ for any $t \geq 0$, and it is nondecreasing in its second argument. This latter condition clearly means that the larger the technological level is, the faster the disease is treated. 

Within this setting, the problem can be then still be written as 
$$\inf_{\xi \in \mathcal{A}}\E\bigg[\int_0^{\tau} e^{-\delta t} C\big(I_t, \xi_t\big) \d t\bigg],$$
where $(I_t)_t$ and $(\xi_t)_t$ are as defined above in this section, but
$$\tau:=\inf\{t \geq 0:\, M_t =1 \}.$$
The independence of $Q$ with respect to $W$ then leads to the equivalent formulation 
$$V(x,y,z,q):=\inf_{\xi \in \mathcal{A}}\E\bigg[\int_0^{\infty} e^{-\delta t - \Lambda_t } C\big(I_t, \xi_t\big) \d t \bigg],$$
which defines a four-dimensional stochastic control problem. Clearly, this problem is much more challenging than \eqref{eq:Probl} and its analysis, requiring different techniques and results, is left for future research. 

Another interesting future work might concern an extension of the previous model in which the social planner can also increase the technological level $Q$ by supporting the research of a vaccination. Assuming that such an investment comes at proportional cost, this problem can be modeled in terms of an intricate stochastic control problem where the transition rate $\lambda$ of the Markov chain is controlled through a singular control. 
\end{remark}
}}
 
In order to tackle Problem \eqref{eq:Probl} with techniques from dynamic programming, it is convenient to keep track of the initial values of $(S_t,I_t,\beta_t)_t$. We therefore set 
$$\mathcal{O}:=\big\{(x,y,z)\in \R^3:\,\, (x,y)\in (0,1)^2,\,\, x+y < 1,\,\, z\in \mathcal{I}\big\},$$
and, when needed, we stress the dependency of $(S_t,I_t,\beta_t)$ with respect to $(x,y,z)\in \mathcal{O}$ and $\xi\in\mathcal{A}$ by writing $(S^{x,y,z;\xi}_t,I^{x,y,z;\xi}_t,\beta^{z;\xi}_t)$. Indeed, due to \eqref{eq:SolSI} and the autonomous nature of \eqref{eq:Z}, we have that $S_t$ and $I_t$ depend on $(x,y,z)$ and on $\xi$ through $\beta_t$, while $\beta_t$ depends only on $z$ and directly on $\xi$. We shall also simply set $(S^{x,y,z}_t,I^{x,y,z}_t,\beta^{z}_t):=(S^{x,y,z;0}_t,I^{x,y,z;0}_t,\beta^{z;0}_t)$ to denote the solutions to \eqref{eq:S}, \eqref{eq:I}, and \eqref{eq:Z} when $\xi \equiv 0$.

Then, for any $(x,y,z)\in \mathcal{O}$, we introduce the problem's value function
\begin{equation}
\label{eq:V}
V(x,y,z) := \inf_{\xi \in \mathcal{A}}\E\bigg[\int_0^{\infty} e^{-\lambda t} C\big(I^{x,y,z;\xi}_t, \xi_t\big) \d t\bigg].
\end{equation}
The latter is well defined given that $C$ is nonnegative. In the next section we will show that $V$ solves the corresponding dynamic programming equation in the classical sense, and we also provide an optimal control in feedback form.

%%%%%%%%%%%%%%%%%%%%%%%%%%%%%%%%%%%

\section{The Solution to the Social Planner Problem}
\label{sec:mainresult}

We introduce the differential operator $\mathcal{L}$ acting on functions belonging to the class $C^{1,1,2}(\R^3)$:\\
\begin{equation}
\label{eq:L}
\big(\mathcal{L}\varphi\big)(x,y,z) := xyz\big(\varphi_y - \varphi_x\big)(x,y,z) - \alpha y \varphi_y(x,y,z) + \frac{1}{2}\sigma^2(z) \varphi_{zz}(x,y,z).
\end{equation}\\
Next,  for any $(y,z,p) \in (0,1) \times \mathcal{I} \times \R$, define\\
\begin{equation}
\label{eq:Cstar}
C^{\star}(y,z,p):= \inf_{\xi \in [0,L]}\Big(C(y,\xi) + b(z,\xi)p\Big),
\end{equation}\\
which is continuous on $[0,1] \times \mathcal{I} \times \R$. Indeed, by Assumptions \ref{assZ}-(ii) and \ref{assC}-(iv), there exists a constant $\overline{K}>0$ such that
\begin{align*}
& |C^{\star}(y',z',p') - C^{\star}(y,z,p)| \leq \sup_{\xi \in [0,L]}\Big(|C(y',\xi)-C(y,\xi)| + |b(z',\xi)-b(z,\xi)||p'| + |b(z,\xi)||p'-p| \Big) \nonumber \\
& \leq \overline{K}\big(|y'-y| + |z'-z||p'| + |p'-p|\big).
\end{align*}

By the dynamic programming principle, we expect that $V$ should solve (in a suitable sense) the Hamilton-Jacobi-Bellman (HJB) equation \\
\begin{equation}
\label{eq:HJB}
\lambda v(x,y,z) = (\mathcal{L}v)(x,y,z) + C^{\star}(y,z,v_z(x,y,z)), \quad (x,y,z) \in \mathcal{O}.
\end{equation}\\
In order to show that $V$ indeed solves \eqref{eq:HJB} in the classical sense, we start with the following important preliminary results. Their proofs are standard in the literature of stochastic control (see, e.g., \cite{Pham, YongZhou1999}), upon employing Assumptions \ref{assZ} and \ref{assC}.
\begin{proposition}
\label{prop:HJB-semiconc}
There exists $K>0$ such that, for each $\textbf{q}:=(x,y,z),\,\textbf{q}':=(x',y',z') \in \mathcal{O}$ 
\begin{itemize}
\item[(i)] 
$0 \leq V(\textbf{q}) \leq K$ {and} $|V(\textbf{q})-V(\textbf{q}')| \leq K|\textbf{q} - \textbf{q}'|$;
i.e., $V$ is bounded and Lipschitz continuous on $\mathcal{O}$;
\item[(ii)]   for any $\mu \in [0,1]$ and for some $K>0$
$$\mu V(\textbf{q}) + (1-\mu) V(\textbf{q}') - V\big(\mu \textbf{q}  + (1-\mu) \textbf{q}'\big) \leq K\mu(1-\mu)|\textbf{q}-\textbf{q}'|^2;$$
i.e., $V$ is semiconcave on $\mathcal{O}$.
\end{itemize}

Moreover, $V$ is a viscosity solution to the HJB equation \eqref{eq:HJB}.
\end{proposition}

\begin{proof}
The first claim of (i) above follows from the fact that $C$ is nonnegative and bounded on $[0,1]^2$; the second claim of (i) is due to Proposition 3.1 in \cite{YongZhou1999}, whose proof can be easily adapted to our stationary setting. Analogously, the semiconcavity property of (ii) can be obtained by arguing as in Proposition 4.5 of \cite{YongZhou1999}. Finally, Theorem 5.2 of \cite{YongZhou1999} (again, easily adapted to our stationary setting) or Proposition 4.3.2-(2) of \cite{Pham} lead to the viscosity property.
\end{proof}

The semiconcavity of $V$, together with the fact that $V$ solves the HJB equation \eqref{eq:HJB} in the viscosity sense, yield the following directional regularity result. 
\begin{proposition}
\label{prop:Vz}
$V_z$ exists continuous on $\mathcal{O}$.
\end{proposition}
\begin{proof}
Let $(x,y,z)\in\mathcal{O}$. By semiconcavity of $V$, there exists the left and right derivatives of $V$ along the direction $z$ at $(x,y,z)$ that we denote, respectively, by $V_z^-(x,y,z), V_z^+(x,y,z)$. Moreover, again by semiconcavity, we have the inequality $V_z^-(x,y,z)\geq V_z^+(x,y,z)$. Assuming, by contradiction, that $V$ is not differentiable with respect to $z$ at $(x,y,z)$ means assuming that $V_z^-(x,y,z)> V_z^+(x,y,z)$. Hence, we can apply Lemma \ref{lemma:app} in the appendix and find a sequence of functions $(\tilde{\varphi}^n)_n\subset C^{2}(\mathcal{O})$ such that 
\begin{equation}
\label{eq:varphinb}
\begin{cases}
\tilde{\varphi}^n(\bar{x},\bar{y},\bar{z})=V(\bar{x},\bar{y},\bar{z}),\\  
\tilde{\varphi}^n\geq V \ \mbox{in a neighborhood of} \ (\bar{x},\bar{y},\bar{z}),\\  |D\tilde{\varphi}^n(\bar{x},\bar{y},\bar{z})|\leq \tilde{L}<\infty,\\ \tilde{\varphi}^n_{zz}(\bar{x},\bar{y},\bar{z})\stackrel{n\to\infty}{\longrightarrow}-\infty.
\end{cases}
\end{equation}
Then, the viscosity subsolution property of $V$ (cf.\ Proposition \ref{prop:HJB-semiconc}) yields
$$
\lambda V(\bar{x},\bar{y},\bar{z})\leq (\mathcal{L}\tilde{\varphi}^n)(\bar{x},\bar{y},\bar{z})+ C^\star(\bar{y},\bar{z},\tilde{\varphi}^n_z(\bar{x},\bar{y},\bar{z})).
$$
Taking the limit as $n\to \infty$ and using \eqref{eq:varphinb} we get a contradiction. We have thus proved that $V_z$ exists at each arbitrary $(x,y,z)\in\mathcal{O}$. 

Now we show that $V_z$ is continuous. Take a sequence $(\bm q^n)_n\subset\mathcal{O}$ such that $\bm q^n\to \bm q\in\mathcal{O}$, and let $\bm\eta^n=(\eta^n_x,\eta^n_y, \eta^n_z)\in D^+V(\bm q^n)$, the latter being nonempty due to the semiconcavity of $V$. Since $V_z$ exists at each point of $\mathcal{O}$, we have $\eta_z^n=V_z(\bm q^n)$. Since $V$ is semiconcave, the supergradient $D^+V$ is locally bounded as a set-valued map, and therefore there exists a subsequence  $(\bm q^{n_k})_k$ such that $\bm\eta^{n_k}\to\bm\eta=(\eta_x,\eta_y, \eta_z)$. By \cite[Prop.\ 3.3.4-(a)]{CS}, we have $\bm\eta\in D^+V(\bm q)$, and again, since $V_z$ exists, we have $\eta_z=V_z(\bm q)$. Hence, we have proved that from any sequence $(\bm q^n)_n\subset\mathcal{O}$ converging to $\bm q$, we can extract a subsequence  $(\bm q^{n_k})_k\subset\mathcal{O}$ such that $V_z(\bm q^{n_k})\to V_z(\bm q)$. By usual arguments on subsequences, the claim follows. 
\end{proof}

We can now prove the main theoretical result of our paper, which ensures that $V$ is actually a classical solution to the HJB equation \eqref{eq:HJB}. In turn, this provides a way to construct an optimal control in feedback form.\footnote{Notice that, in order to define a candidate optimal control in feedback form, one actually only needs the existence of the derivative $V_z$. For instance, in the deterministic problem tackled in \cite{FedericoTacconi} only the regularity of the directional derivative is exploited to prove a verification theorem in the context of viscosity solutions. However, here we can improve the regularity of $V$ due to the stochastic nature of our problem, and therefore prove a classical verification theorem.}
\begin{theorem}
\label{thm:main}
The following holds: 
\begin{itemize}
\item[(i)]
$V \in C^{2}(\mathcal{O})$ and solves the HJB equation \eqref{eq:HJB} in the classical sense. 
\item[(ii)]
Let 
\begin{equation}
\label{eq:OCrule}
\widehat{\xi}(x,y,z):=\argmin_{\xi\in[0,L]}\Big(C(y,\xi) - b(z,\xi)V_z(x,y,z)\Big), \ \ \ \ \ (x,y,z)\in \mathcal{O}.
\end{equation}
If the system of equations\footnote{Since $b$ is bounded, by the method of Girsanov's transformation, the system has a weak solution, which is also unique in law (see  \cite[Ch.\,5,\,Propositions\,3.6 and 3.10]{KS} and  also \cite[Ch.\,5,\,Remark\,3.7]{KS}). For the sake of brevity, we do not investigate further existence and uniqueness of strong solutions, even if this might be done by employing finer results (e.g., see the seminal paper \cite{V}).}
\begin{equation}
\label{system-OC}
\begin{cases}
\d S_t = - \beta_t S_t I_t \d t, \ \ \ \ \ \ \ \ \ \ \ \ \ \ \ \ \ \ \ \ \ \ \ \ \  \qquad S_0= x,\\
\d I_t = \big(\beta_t S_t I_t - \alpha I_t\big) \d t, \ \ \ \ \ \ \ \ \ \ \ \ \ \ \ \ \ \ \ \ \ \ \ \ I_0= y,\\
\d \beta_t = b(\beta_t,\widehat\xi(S_t,I_t,\beta_t)) \d t + \sigma(\beta_t)\d W_t,  \ \ \ \ \  \beta_0= z,
\end{cases}
\end{equation}
admits a unique strong solution $(S^\star_t, I^\star_t, \beta^{\star}_t)_t$, then the control\\
\begin{equation}
\label{eq:OC}
\xi^{\star}_t:=\widehat\xi\big(S^\star_t, I^\star_t, \beta^{\star}_t\big),
\end{equation}\\
is optimal for \eqref{eq:V} and $(\beta^{\star}_t)_t$ is the optimally controlled transmission rate; that is,
$$V(x,y,z)=\E\bigg[\int_0^{\infty} e^{-\lambda t} C\big(I^{\star}_t, \xi^{\star}_t\big) \d t\bigg].$$
\end{itemize}
\end{theorem}

\begin{proof}
\emph{Proof of (i) - Step 1.}  Recall \eqref{eq:Cstar} and define $F(x,y,z):=C^{\star}(y,z,V_z(x,y,z))$. Due to Proposition \ref{prop:Vz} and the continuity of $C^{\star}$ on $(0,1) \times \mathcal{I} \times \R$, we have that $F$ is continuous on $\mathcal{O}$. Moreover, since $C$ is bounded on $[0,1]\times[0,L]$, $V_z$ is bounded on $\mathcal{O}$ by Proposition \ref{prop:HJB-semiconc}-(i), and $b(\cdot,\xi)$ is bounded (cf.\ Assumption \ref{assZ}-(ii)), there exists $K>0$ such that
\begin{equation}
\label{eq:growF}
|F(x,y,z)| \leq K, \quad \forall (x,y,z) \in \mathcal{O}.
\end{equation}

Set now
\begin{equation}
\label{eq:v}
v(x,y,z):=\E\bigg[\int_0^{\infty}e^{-\lambda t} F\big(S^{x,y,z}_t, I^{x,y,z}_t, \beta^z_t\big) \d t\bigg], \quad (x,y,z)\in \mathcal{O}.
\end{equation}

%{{GIO:ora che F sta bounded, questa parte qui sotto non serve e non serve neanche la stima a la Krylov...}}
%which is well defined and finite because
%$$\E\bigg[\int_0^{\infty}e^{-\lambda t} F\big(S^{x,y,z}_t, I^{x,y,z}_t, \beta^z_t\big) \d t\bigg] \leq \int_0^{\infty}e^{-\lambda t} \big(1 + \E\big[|\beta^{z}_t|\big]\big) \d t < \infty,$$
%where the first inequality is due to \eqref{eq:growF}, and the last one follows from \eqref{eq:momZ} and the fact that $\lambda > \kappa_1 \vee 0$ by assumption.

Although not uniformly elliptic, the differential operator $\mathcal{L}$ defined in \eqref{eq:L} is hypoelliptic, meaning that the so-called H\"ormander's condition is satisfied (cf.\ the proof of Proposition \ref{prop:A1} in the appendix and equation \eqref{eq:Hormander} therein). In fact, by Proposition \ref{prop:A1} in the appendix, for any $\textbf{q}:=(x,y,z)\in \mathcal{O}$ the (uncontrolled) process $(\textbf{Q}^{\textbf{q}}_t)_t:=(S^{x,y,z}_t, I^{x,y,z}_t, \beta^z_t)$ admits a transition density $p(t,\textbf{q},\cdot)$, $t>0$, which is absolutely continuous with respect to the Lebesgue measure in $\R^3$, infinitely many times differentiable, and satisfying the Gaussian estimates \eqref{eq:Gauss1} and \eqref{eq:Gauss2}. As a consequence, by Fubini's theorem we can write
$$v(x,y,z)=\int_0^{\infty}e^{-\lambda t} \Big(\int_{\mathcal{O}}F\big(x', y', z'\big) p(t, x, y, z; x', y', z') \d x' \d y' \d z' \Big) \d t,$$
and recalling \eqref{eq:growF}, and applying the dominated convergence theorem, one shows that $v \in C^{2}(\mathcal{O})$.

For $(x,y,z) \in \mathcal{O}$, let now $\tau_n:=\inf\{t\geq0:\,|(S^{x,y,z}_t, I^{x,y,z}_t, \beta^z_t)|\geq n\}$, $n\in \mathbb{N}$, and notice that the strong Markov property yields
$$e^{-\lambda (t \wedge \tau_n)}v(S^{x,y,z}_{t\wedge\tau_n},I^{x,y,z}_{t\wedge\tau_n},\beta^{z}_{t\wedge\tau_n}) + \int_0^{t\wedge\tau_n} F\big(S^{x,y,z}_u, I^{x,y,z}_u, \beta^z_u\big) \d u = \E\bigg[\int_0^{\infty}e^{-\lambda t} F\big(S^{x,y,z}_t, I^{x,y,z}_t, \beta^z_t\big) \d t\,\Big|\, \mathcal{F}_{t\wedge\tau_n}\bigg].$$
Since $v \in C^{2}(\mathcal{O})$, we can apply It\^o's formula to the first addend on the left-hand side of the latter, take expectations, observe that the stochastic integral has zero mean (by definition of $\tau_n$ and the fact that $v_x$ is continuous), and finally find
\begin{equation}
\label{eq:eqv-1}
\E\bigg[\int_0^{t \wedge \tau_n}e^{-\lambda u} \big(\mathcal{L}v + F - \lambda v\big)\big(S^{x,y,z}_u, I^{x,y,z}_u, \beta^z_u\big) \d u\bigg] + v(x,y,z) = \E\bigg[\int_0^{\infty}e^{-\lambda t} F\big(S^{x,y,z}_t, I^{x,y,z}_t, \beta^z_t\big) \d t\bigg];
\end{equation}
that is, by \eqref{eq:v},
$$\E\bigg[\int_0^{t \wedge \tau_n}e^{-\lambda u} \big(\mathcal{L}v + F - \lambda v\big)\big(S^{x,y,z}_u, I^{x,y,z}_u, \beta^z_u\big) \d u\bigg] = 0.$$
Dividing now both left and right-hand sides of the latter by $t$, invoking the (integral) mean-value theorem, letting $t\downarrow 0$, and using that $t\mapsto (S^{x,y,z}_t, I^{x,y,z}_t, \beta^z_t)$ is continuous, we find that $v$ is a classical solution to 
\begin{equation}
\label{eq:eqv-2}
\lambda \varphi = \mathcal{L}\varphi + F \quad \text{on} \quad \mathcal{O}.
\end{equation}

\emph{Proof of (i) - Step 2.} Let $(x,y,z) \in \mathcal{O}$, and $(\mathcal{K}_n)_n$ be an increasing sequence of open bounded subsets of $\mathcal{O}$ such that $\bigcup_{n\in\mathbb{N}}\mathcal{K}_n=\mathcal{O}$. Defining the stopping time $$\rho_n:=\inf\{t\geq0:\,(S^{x,y,z}_t, I^{x,y,z}_t, \beta^z_t\big) \notin \mathcal{K}_n\}, \ \ \ \ ,n\in \mathbb{N},$$ we set
\begin{equation}
\label{eq:defvhat}
\widehat{v}_n(x,y,z):= \E\bigg[\int_0^{\rho_n} F\big(S^{x,y,z}_u, I^{x,y,z}_u, \beta^z_u\big) \d u + e^{-\lambda \rho_n}V\big(S^{x,y,z}_{\rho_n},I^{x,y,z}_{\rho_n},\beta^{z}_{\rho_n}\big)\bigg].
\end{equation}
If $(x,y,z) \notin \mathcal{K}_n$, then $\widehat{v}_n(x,y,z)=V(x,y,z)$ as $\rho_n=0$ a.s. Take then $(x,y,z) \in \mathcal{K}_n$. By the same arguments as in Step 1 and considering that $V$ is  continuous  on $\mathcal{K}_n$, the function  $\widehat{v}_n$ is a solution to
\begin{equation}
\label{eq:eqvhat-1}
\lambda \varphi = \mathcal{L}\varphi + F, \quad \text{on} \quad \mathcal{K}_n, \qquad \varphi=V \quad \text{on} \quad \partial\mathcal{K}_n.
\end{equation}
Since also $V$ is a viscosity solution to the same equation and since uniqueness of viscosity solution holds for such a problem (cf., e.g., \cite{CIL}), we have $\widehat{v}_n = V$ on $\overline{\mathcal{K}}_n$. Because $\rho_n \uparrow \infty$ for $n \uparrow \infty$ (as the boundary of $\mathcal{O}$ is unattainable for $\big(S^{x,y,z}_t, I^{x,y,z}_t, \beta^z_t\big)$), by taking limits as $n \uparrow \infty$ in \eqref{eq:defvhat} we find that
$$V(x,y,z) = \lim_{n\uparrow\infty}\widehat{v}_n(x,y,z) = v(x,y,z), \quad (x,y,z) \in \mathcal{O},$$
where the last equality follows by dominated convergence upon recalling that $V$ is bounded. But then $V=v$ on $\mathcal{O}$, and therefore $V \in C^{2}(\mathcal{O})$ and solves \eqref{eq:eqv-2} by \emph{Step 1}. That is, $V$ is a classical solution to the HJB equation \eqref{eq:HJB}.
\vspace{0.25cm}

\emph{Proof of (ii).} The optimality of \eqref{eq:OC} follows by a standard verification theorem based on an application of It\^o's formula and the proved regularity of $V$ (see, e.g., Chapter 3.5 in \cite{Pham}).
\end{proof}

%%%%%%%%%%%%%%%%%%%%%%%%%%%%%%%%%%%

\section{A Case Study with Numerical Illustrations}
\label{sec:numerics}

In this section we illustrate numerically the results of our model, with the aim of providing qualitative properties of the optimal containment policies in a case study.

We use the mean-reverting model for the dynamics of $\beta$, i.e. 
\begin{equation}
\label{eq:ZOUbis}
\d \beta_t = \vartheta\Big(\widehat{\beta}\big(L - \xi_t\big) - \beta_t\Big) \d t + \sigma\beta_t(\gamma - \beta_t)\d W_t, \quad t>0, \qquad \beta_0= z \in (0,\gamma),
\end{equation}
for some $L, \vartheta, \gamma, \sigma>0$, $\widehat{\beta} \in (0,\gamma)$. {{Notice that such a choice of the dynamics of $\beta$ fulfills all the requirements of Assumption \ref{assZ} (see Proposition \ref{prop:betadyn} in Appendix \ref{sec:app})}}. Moreover, we assume that the social planner has a quadratic cost function of the form
\begin{equation}
\label{costexample}
C(y, \xi)=\left(\frac{y}{\bar{y}}\right)^2 +\frac{1}{2}\xi^2.
\end{equation}
{{The latter can be interpreted as a Taylor approximation of any smooth, convex, separable cost function with global minimum in $(0,0)$.}} In \eqref{costexample}, $\bar{y} \in (0,1)$ represents, e.g., the maximal percentage of infected people that the health-care system can handle. 

Notice that in this case for any $(x,y,z)\in \mathcal{O}$ one has (cf.\ \eqref{eq:OCrule})
\begin{equation}
\label{eq:OCrule-bis}
\widehat{\xi}(x,y,z) = \begin{cases}
L, \ \ \ \ \ \ \ \ \ \ \ \ \ \ \ \ \ \ \ \ \ \ \ \ \ \ \text{if}\,\, V_z(x,y,z) > \frac{L}{\vartheta \widehat{\beta}},\\
\vartheta \widehat{\beta}V_z(x,y,z), \ \ \ \ \ \ \ \ \ \ \ \text{if}\,\, V_z(x,y,z) \in [0, \frac{L}{\vartheta \widehat{\beta}}], \\
0,  \ \ \ \ \ \ \ \ \ \ \ \ \ \ \ \ \ \ \ \ \ \ \ \ \ \ \text{if}\,\, V_z(x,y,z) < 0.
\end{cases}
\end{equation}

Our numerics is based on a recursion on the nonlinear equation
$$\big(\lambda - \mathcal{L}\big)v(x,y,z) = C^{\star}(y,z,v_z(x,y,z)), \quad (x,y,z) \in \mathcal{O},$$
which is solved by the value function in the classical sense (cf.\ Theorem \ref{thm:main}).
Namely, starting from $v^{[0]} \equiv 0$ we use the recursive algorithm:
$$\big(\lambda - \mathcal{L}\big)v^{[n+1]} = C^{\star}(y,z,v^{[n]}_z), \quad n \geq 1$$
and those equations are solved by Montecarlo methods based on the Feynmann-Kac formula
$$v^{[n+1]}(x,y,z) = \E\bigg[\int_0^{\infty}e^{-\lambda t} C^{\star}\big(I^{x,y,z}_t, \beta^{z}_t, v^{[n]}_z(S^{x,y,z}_t, I^{x,y,z}_t, \beta^{z}_t)\big) \d t\bigg], \quad (x,y,z)\in \mathcal{O}.$$
Such an approach is needed because of the lack of appropriate boundary conditions on the HJB equation, as the boundary $\partial\mathcal{O}$ is unattainable for the underlying controlled dynamical system.

{{We take a day as a unit of time.}} In our experiments we assume that the average length of an infection equals $18$ days, so that $\alpha=\frac{1}{18}$ (see also \cite{Alvarez}, \cite{Erhan}, and \cite{KruseStrack}), the level of the maximal possible transmission rate of the disease is $\gamma=0.16$, the natural transmission rate of the disease is $\widehat{\beta}=0.1$, towards which the transmission rate $(\beta_t)_t$ reverts at rate $\vartheta=0.1$ when $\xi\equiv0$, {{and $\sigma=1$, so that the fluctuations of $(\beta_t)_t$ are (at most) of order $10^{-2}$}}. Furthermore, we set $\lambda=1/365$\footnote{{{Our choice of the value of $\lambda=\lambda_o + \delta$ can be justified by assuming that it takes at least a year to develop a vaccine (i.e.\ $1/\lambda_o \geq 365$) and that the intertemporal discount rate of the social planner $\delta$ is negligible with respect to vaccination discovery rate. Indeed, a typical value for the annual discount rate $\delta$ is $5\%$ which is clearly such that $\frac{0.05}{365} \ll \frac{1}{365}$.}}}, and we fix $\bar{y}={0.1}$ in \eqref{costexample}. Finally, in all simulations we assume that at day zero about $1\%$ of the population is infected. 

{{In all the subsequent pictures we show the mean paths of the considered quantities, with their $95\%$ confidence interval. The Montecarlo average has been performed by employing $6000$ independent simulations.}}

In Section \ref{optimal}, we compare the optimal social planner policy with the case of no restrictions; in Section \ref{limitedContainment} we consider strategies in which the containment measures are limited to a fixed  percentage $L \in [0,1]$ and provide a comparison between them; in Section \ref{sec:sigma} we study the effect of the fluctuations of the transmission rate on the problem's solution.

%%%%%%%%%%%%%%%%%%%%%%%%%%%%%%%

\subsection{The Optimal Social Planner Policy}
\label{optimal} 

We compare the optimal social planner policy with the case of no restrictions (see Figure \ref{fig1}). In the optimal social planner policy severe lockdown measures (larger than $40\%$) are imposed for a period of circa 63 days, starting on day 79; then, it follows a gradual reopening phase. 
The final percentage of recovered individuals is about $50\%$, in contrast to $72\%$ which is the total percentage of recovered individuals in the case of no restrictions.

Furthermore, the cases of optimal lockdown and no lockdown show a substantial difference in the evolution of the reproduction number $\mathcal{R}_t:=\frac{\beta_t}{\alpha}$: in the case of lockdown policies at work, in the most restrictive period, the latter is significantly decreased to values around $0.6$. Another relevant quantity to analyze is $\mathcal{R}_t S_t$. Indeed, recalling \eqref{eq:I}, it is easy to see that the percentage of infected naturally decreases at exponential rate $\alpha(\mathcal{R}_t S_t - 1)$ if $\mathcal{R}_t S_t$ is maintained strictly below $1$. We observe that, under the suboptimal action ``no lockdown'', $\mathcal{R}_t S_t$ lies below one from day $85$ on. On the other hand, the optimal containment policy is such that $\mathcal{R}_t S_t <1$ from day $75$ on. As a consequence, $\mathcal{R}_t$ can be let oscillate strictly above one (actually, around $1.7$) during the final phase of partial reopening so that the negative impact of lockdowns on the economic growth can be partially dammed.

\begin{figure}[htb]
\begin{center}
\begin{tabular}{cccccc}
\includegraphics[height=2.8cm]{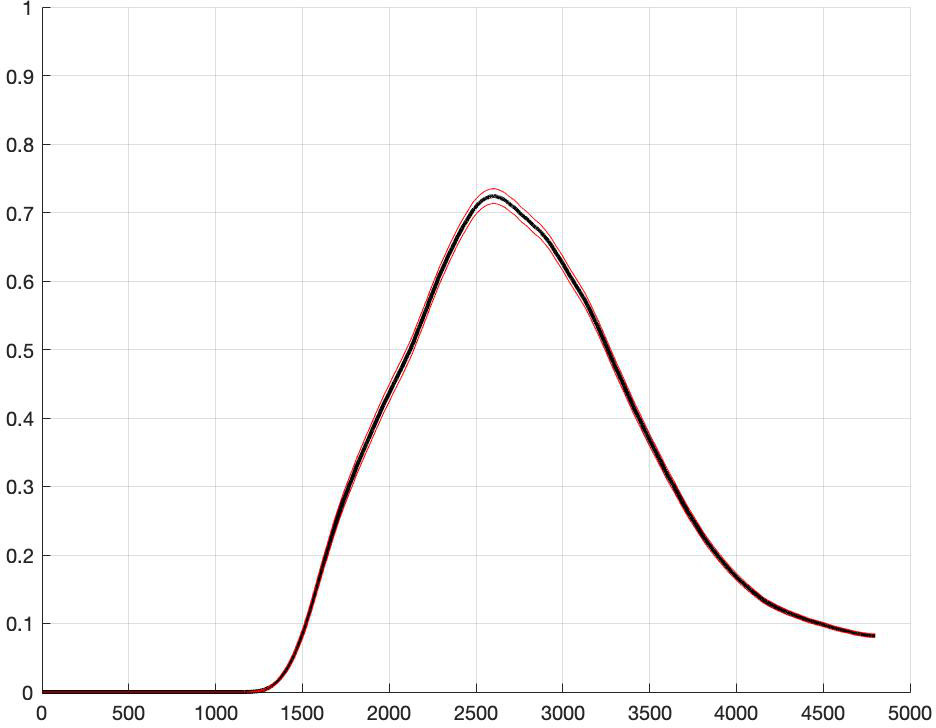}
 &
 \includegraphics[height=2.8cm]{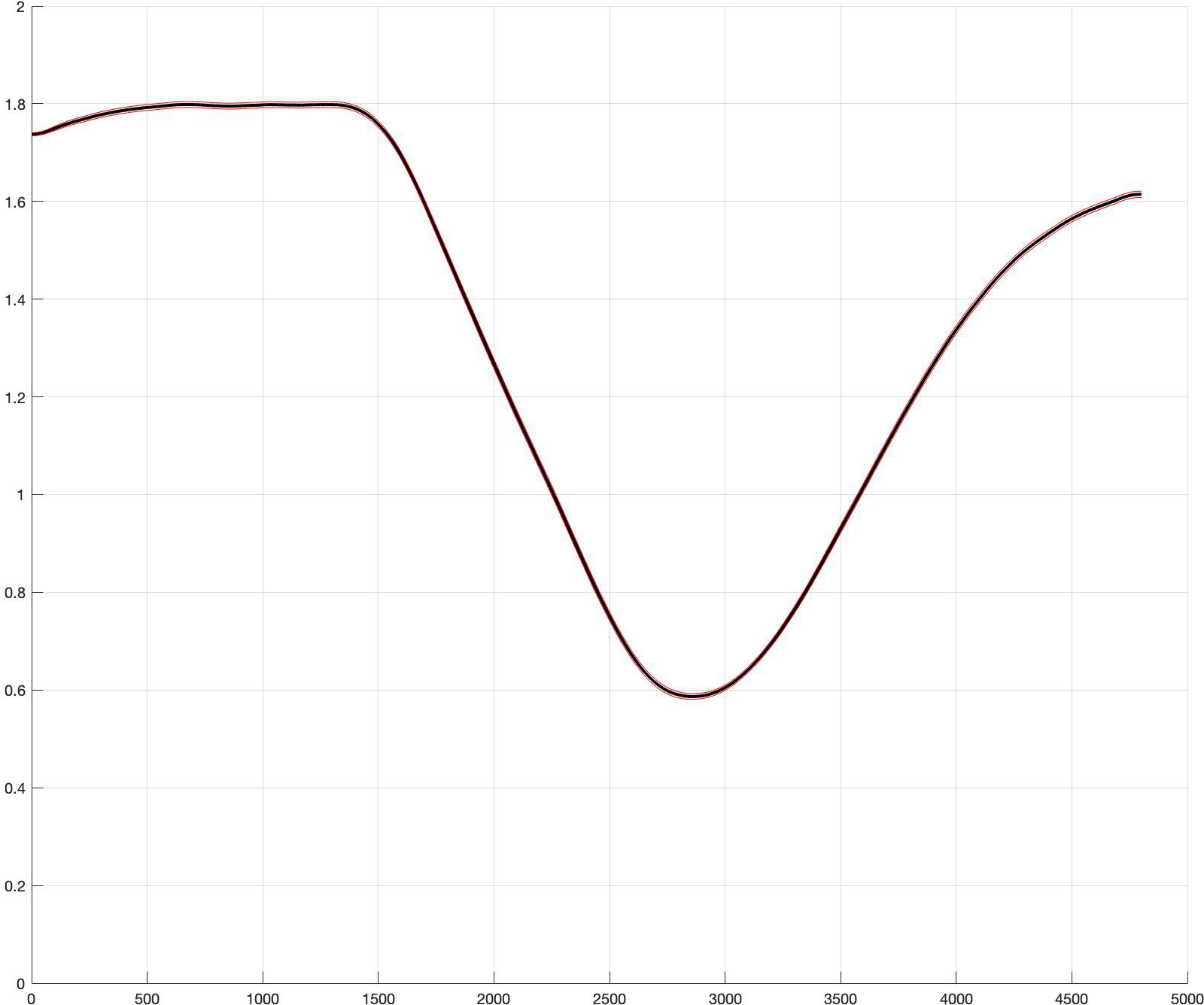} &  
 \includegraphics[height=2.8cm]{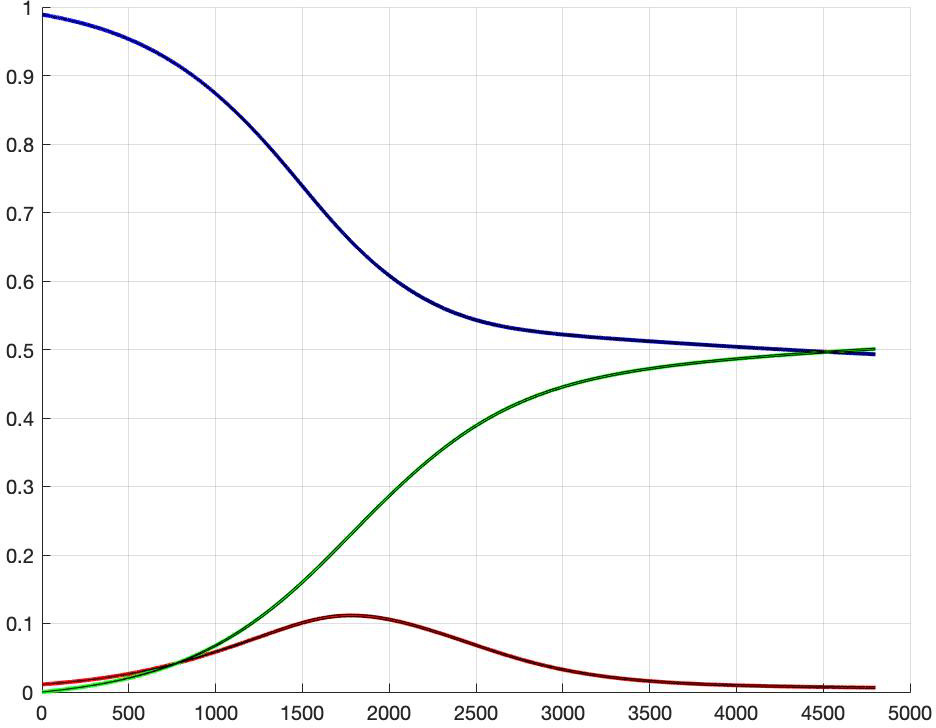}  &
\includegraphics[height=2.8cm]{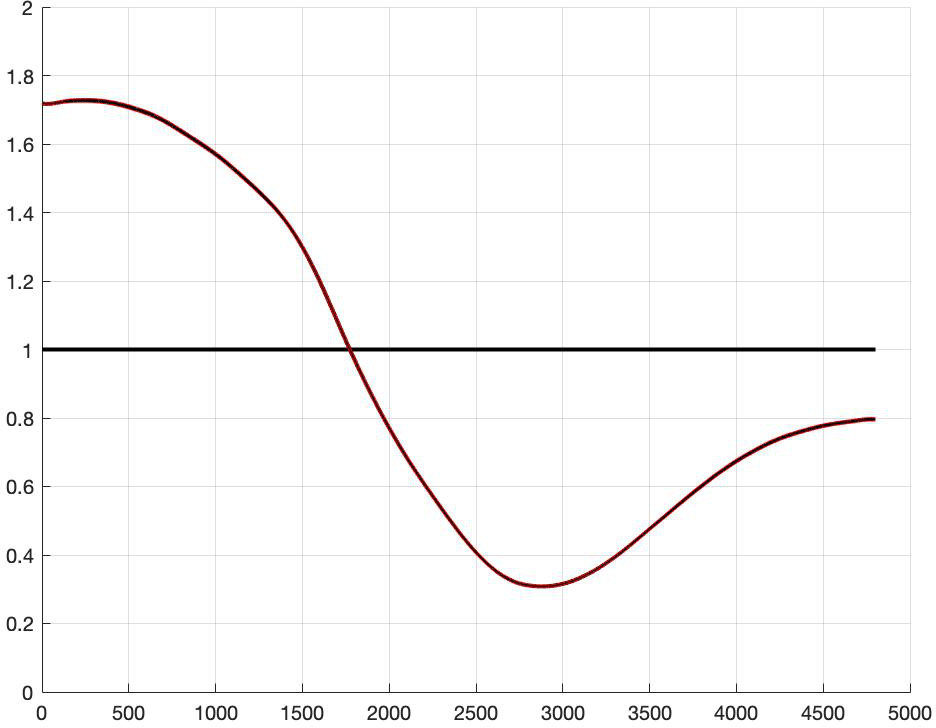}
\\ \\
%(a) & & (b)\\ 
\includegraphics[height=2.8cm]{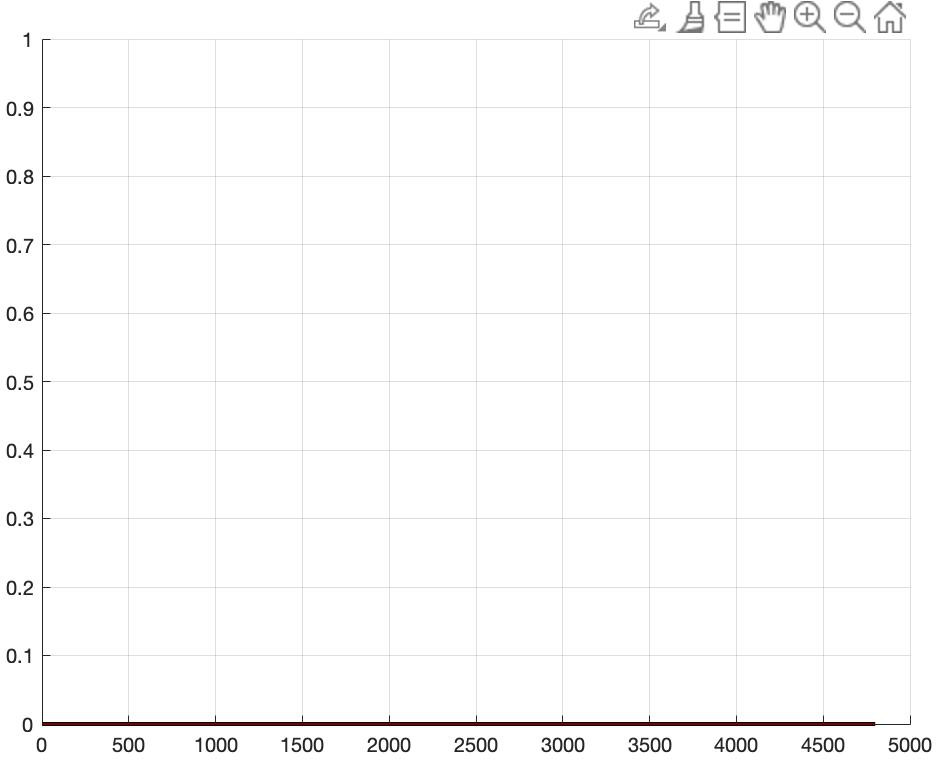}
 & 
\includegraphics[height=2.8cm]{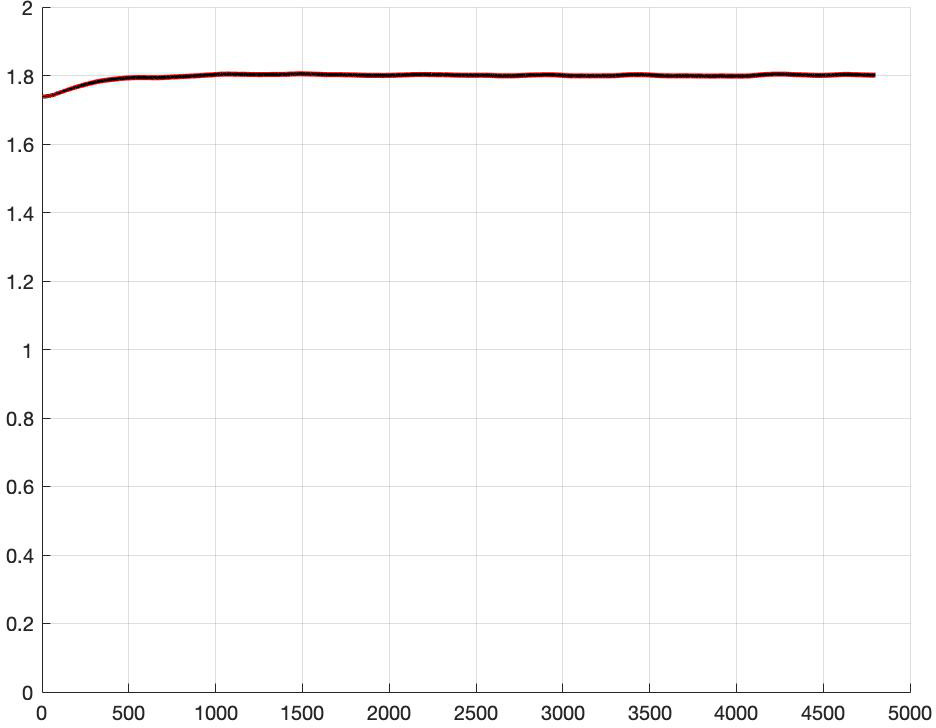} &
\includegraphics[height=2.8cm]{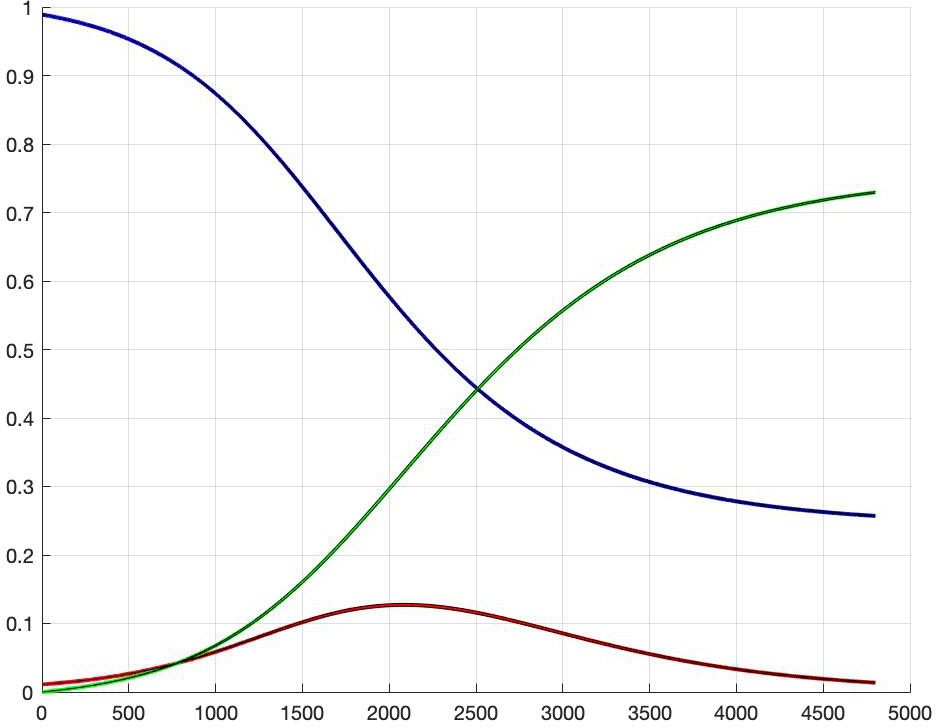} &  
\includegraphics[height=2.8cm]{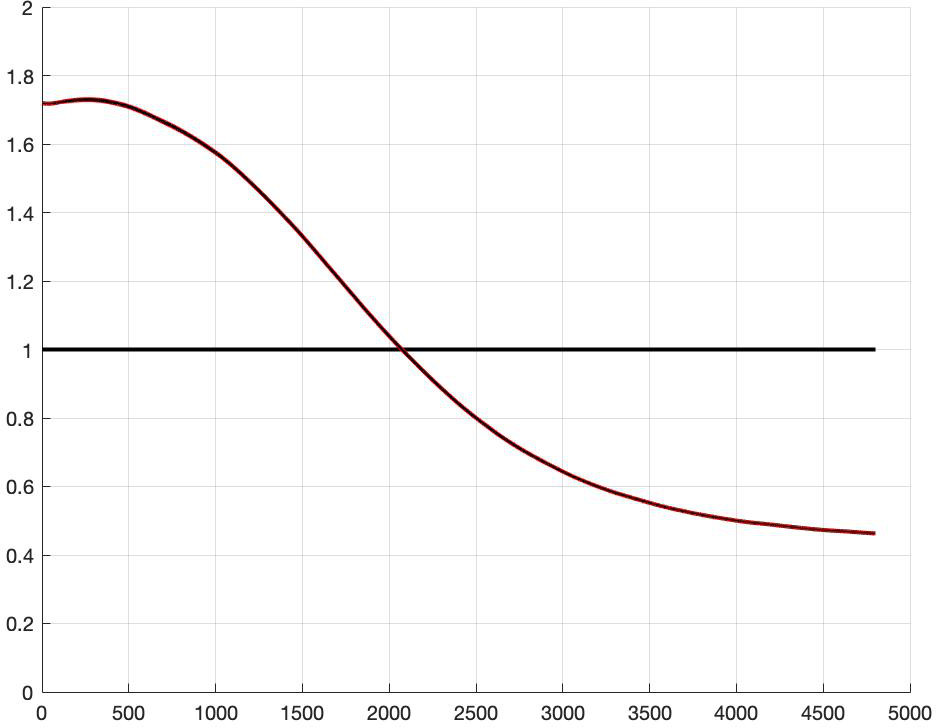}\\
%(c) & & (d)
\end{tabular}
\caption{Comparison between the optimal social planner policy (upper panel) and the case of no restrictions (lower panel).
The figures in the first column show the (average) evolution of the containment policy through the value of the optimal control $\xi_t$; the ones in the second column show  
the (average) evolution of the instantaneous reproduction number $\mathcal{R}_t=\frac{\beta_t}{\alpha}$; 
the ones in the third column show (average) evolution of the percentage of susceptible (in blue), infected (in red) and recovered (in green) individuals; the ones in the fourth column show the (average) evolution of the product $\mathcal{R}_t\cdot S_t$.}
\label{fig1}
\end{center}
\end{figure}

%%%%%%%%%%%%%%%%%%%%%%%%%%%%

\subsection{The Optimal Social Planner Policy with Limited Containment}
\label{limitedContainment}

In many countries, a vigorous lockdown could not always be feasible, especially for long periods. Further, as pointed out by recent literature (for instance see \cite{Asprietal}), gradual policies of longer duration but more moderate containment exhibit large welfare benefits comparable to the ones obtained by a drastic lockdown. For this reason, we consider a strategy in which the containment measures are limited to a fixed percentage $L \in [0,1]$. Notice that $L=0.7$ in \cite{Alvarez}, $L=\{0.7, 1\}$ in \cite{Acemoglu} and \cite{Erhan}. A comparison of the optimal social planner policy with limited containment $L \in \{0.2, 0.4, 0.6, 0.8\}$ is shown in Figure \ref{fig2} and a 
summary is contained in Table \ref{tab1}. 
\begin{figure}[htb]
\begin{center}
\begin{tabular}{ccccccc}
\includegraphics[height=2.8cm]{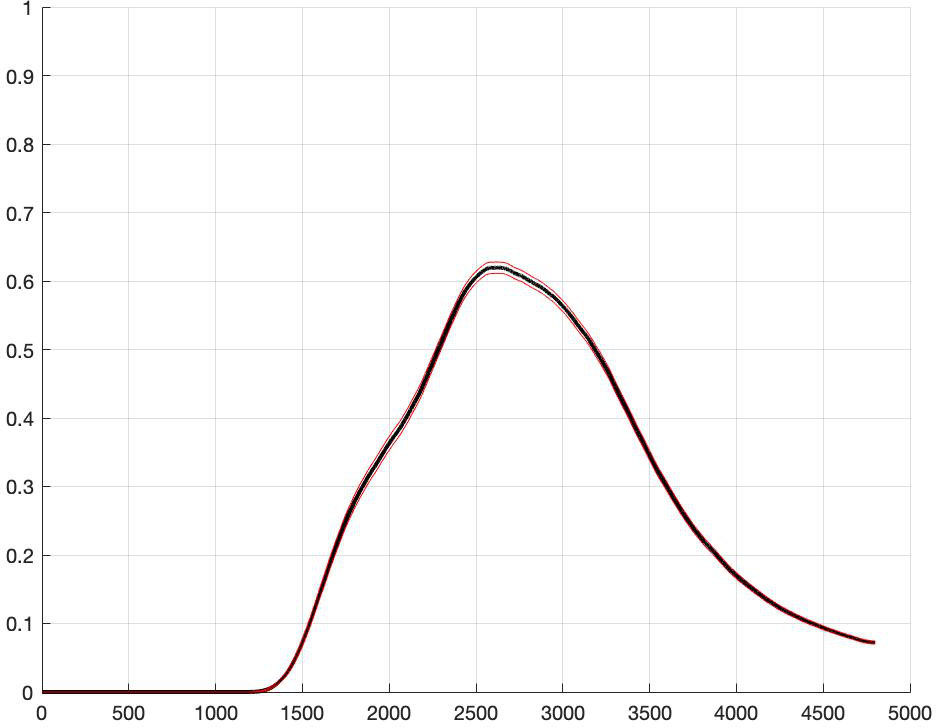} & &
\includegraphics[height=2.8cm]{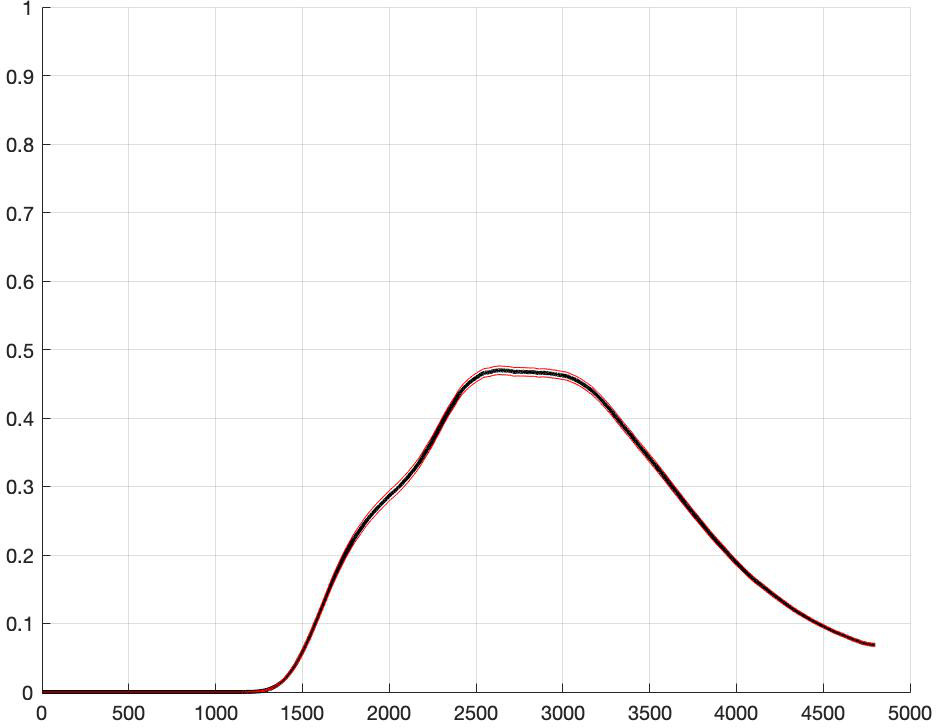} &  &
\includegraphics[height=2.8cm]{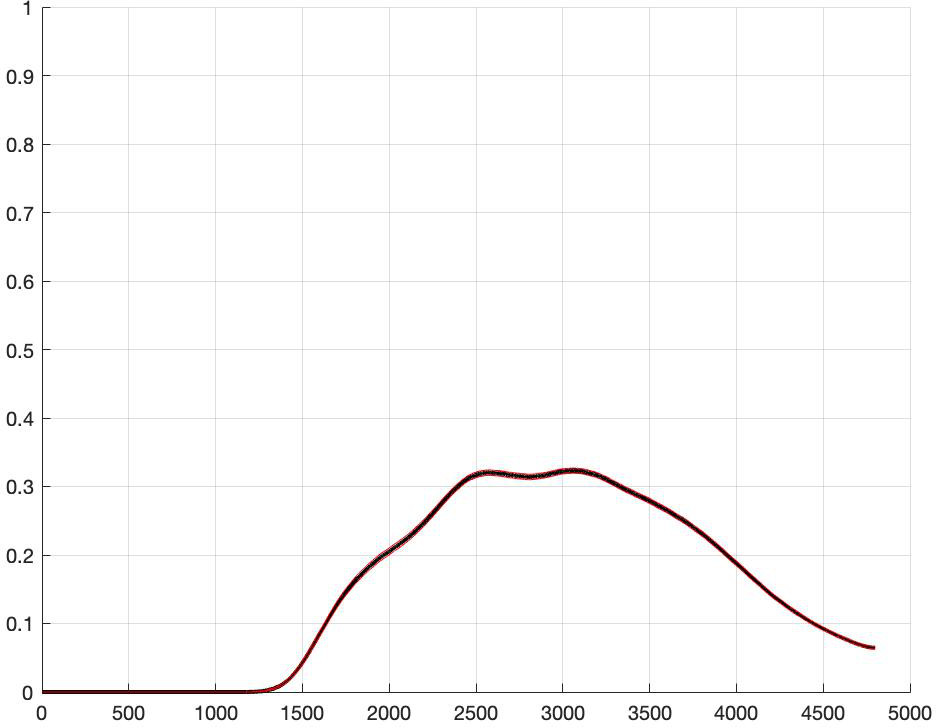} &  &
\includegraphics[height=2.8cm]{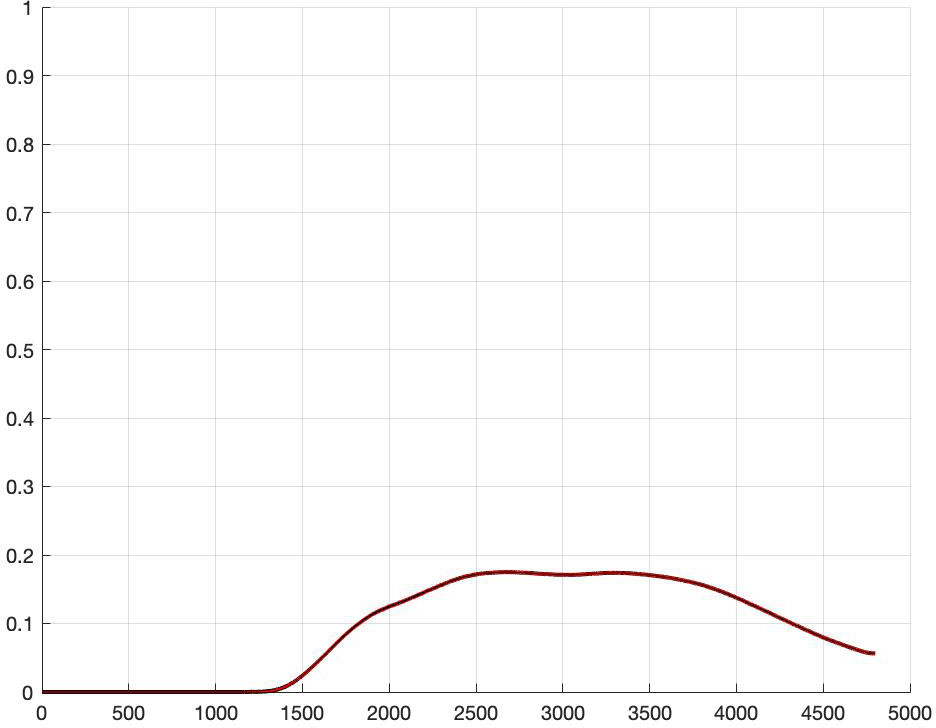} \\ \\
%(a) $L=0.2$ & & (b) $L=0.4$ & & (c) $L=0.6$ & & (d) $L=0.8$\\ \\
\includegraphics[height=2.8cm]{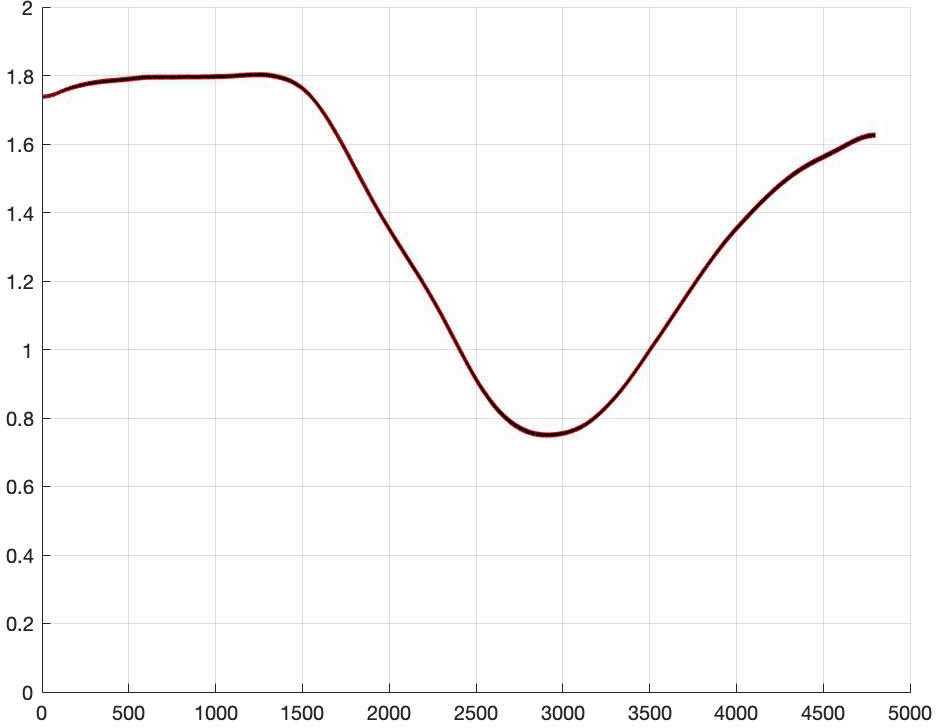} & &
\includegraphics[height=2.8cm]{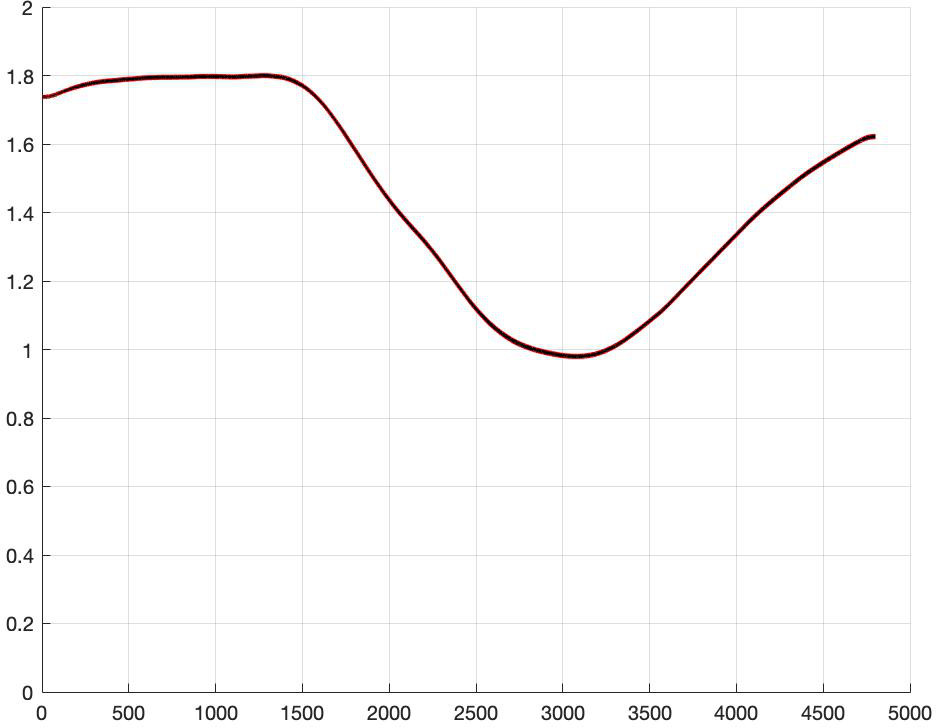} &  &
\includegraphics[height=2.8cm]{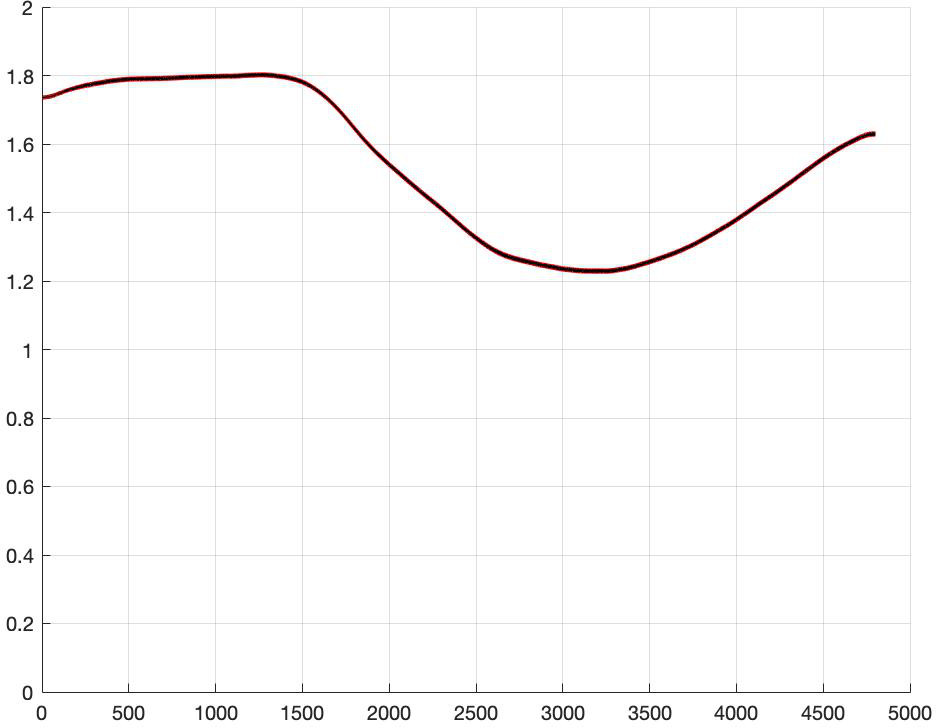} &  &
\includegraphics[height=2.8cm]{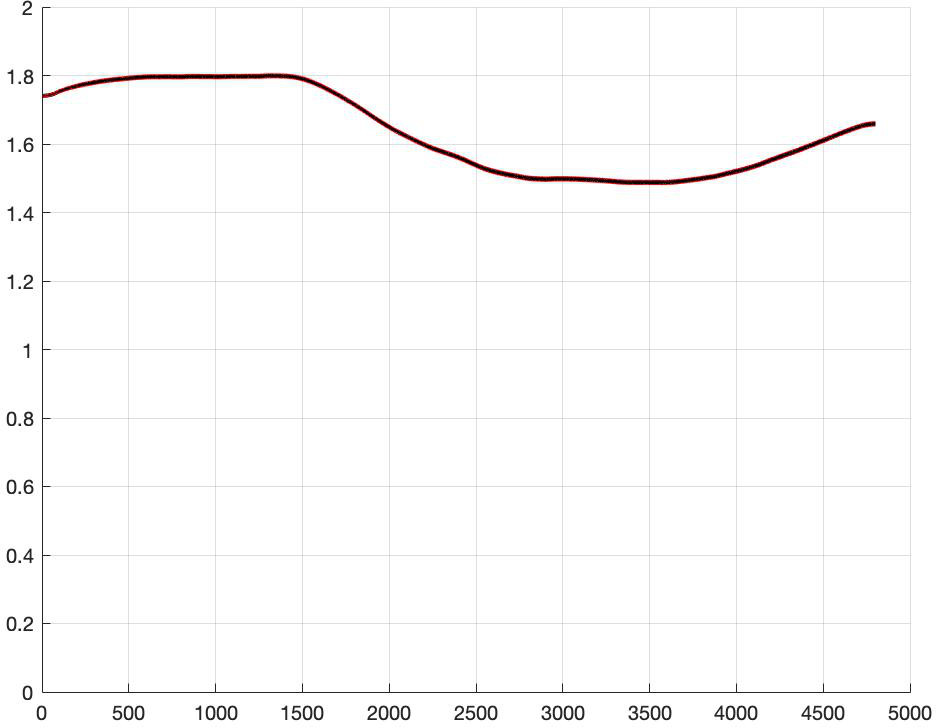} \\ \\
%(a) $L=0.2$ & & (b) $L=0.4$ & & (c) $L=0.6$ & & (d) $L=0.8$\\ \\
\includegraphics[height=2.8cm]{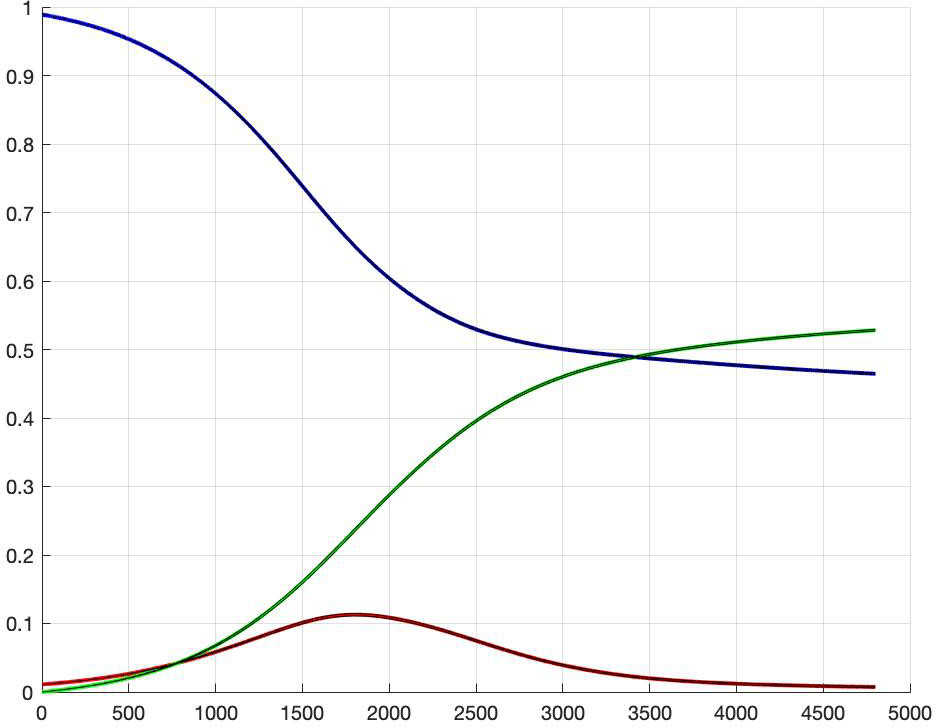} & &
\includegraphics[height=2.8cm]{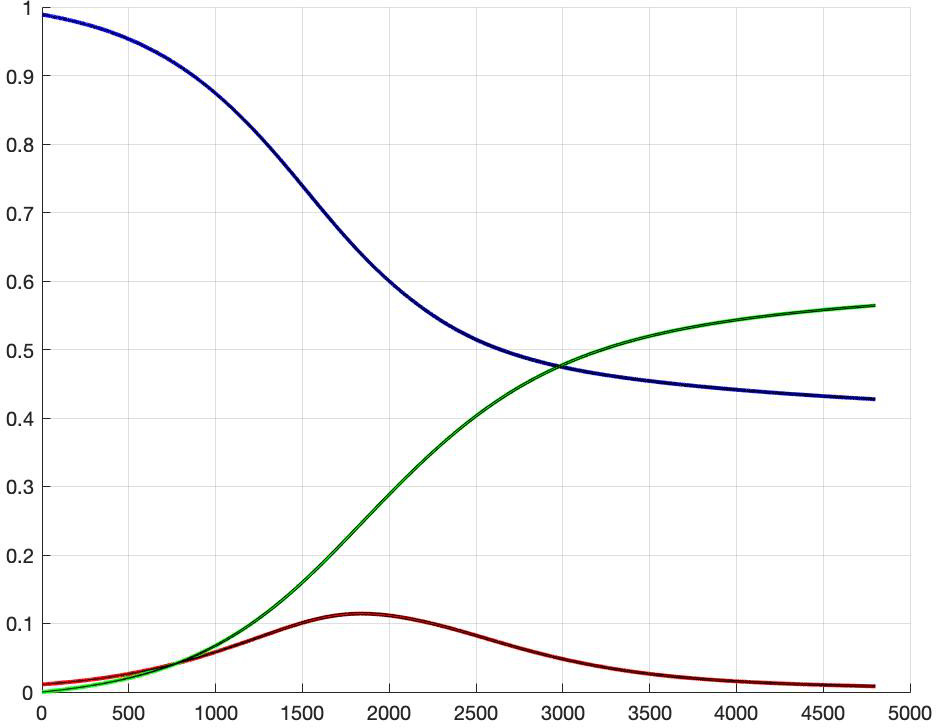} &  &
\includegraphics[height=2.8cm]{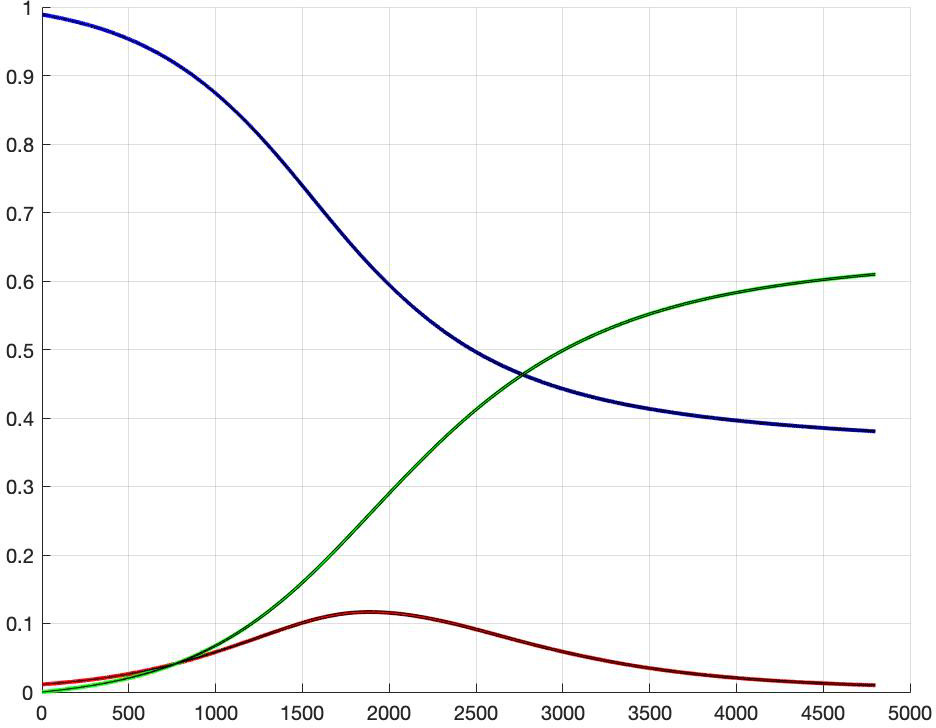} &  &
\includegraphics[height=2.8cm]{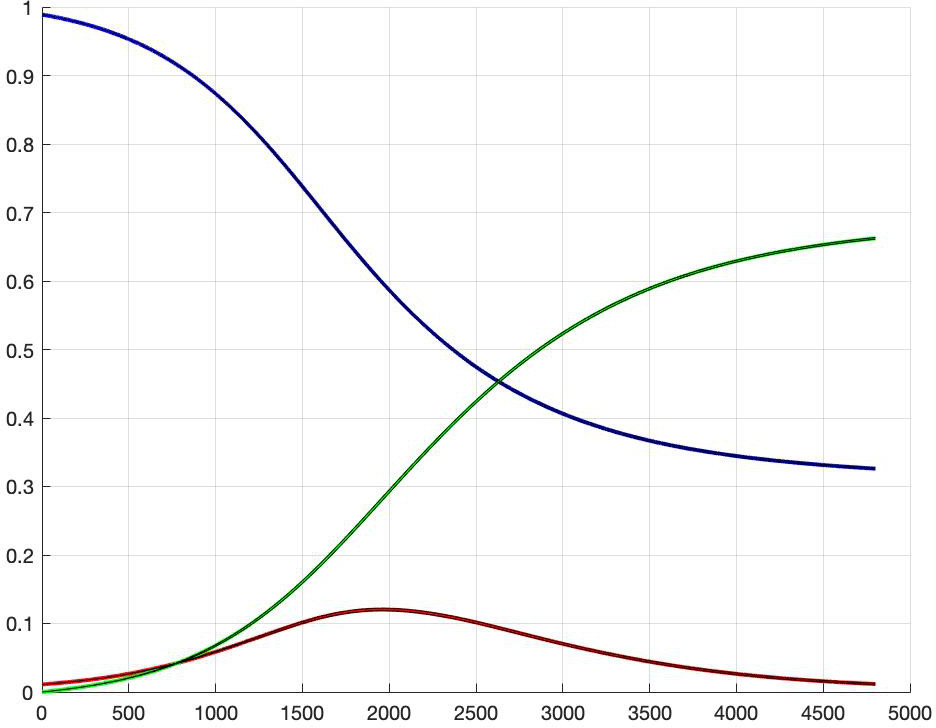} \\ \\
\includegraphics[height=2.8cm]{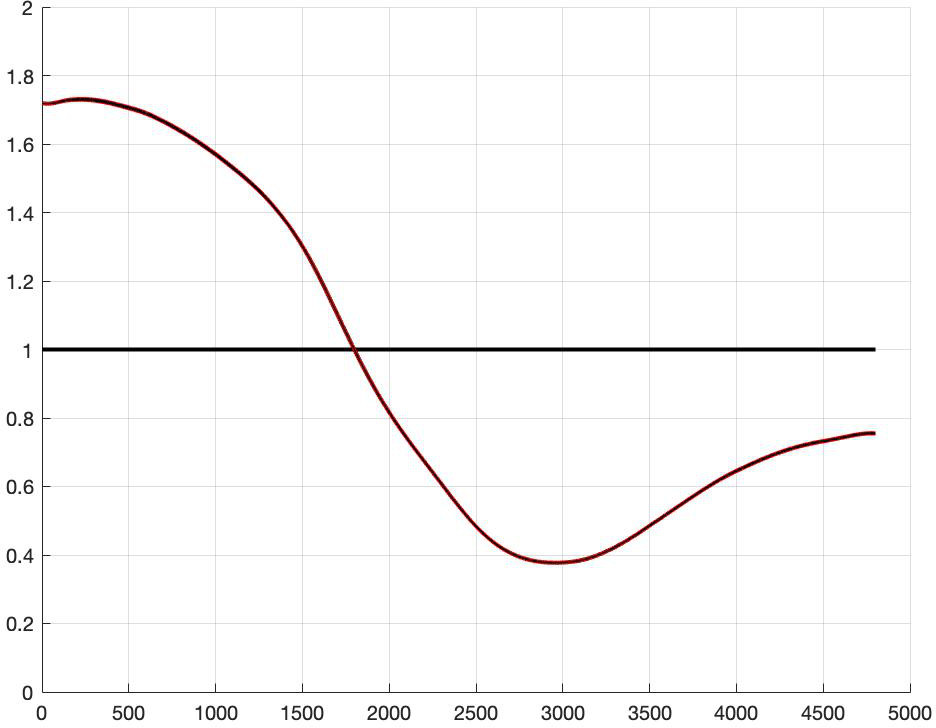} & &
\includegraphics[height=2.8cm]{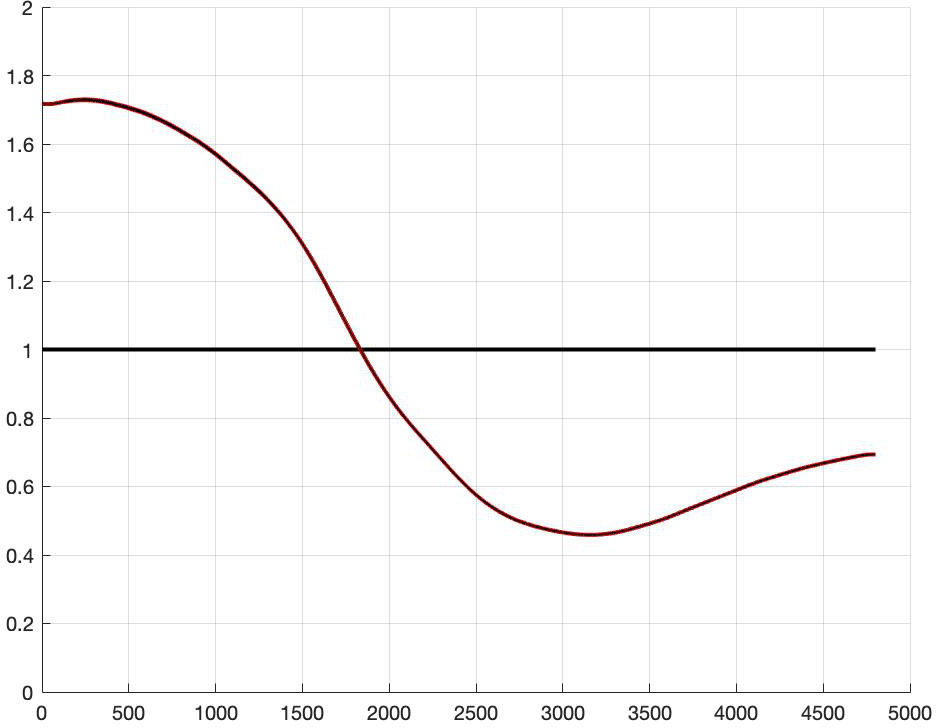} &  &
\includegraphics[height=2.8cm]{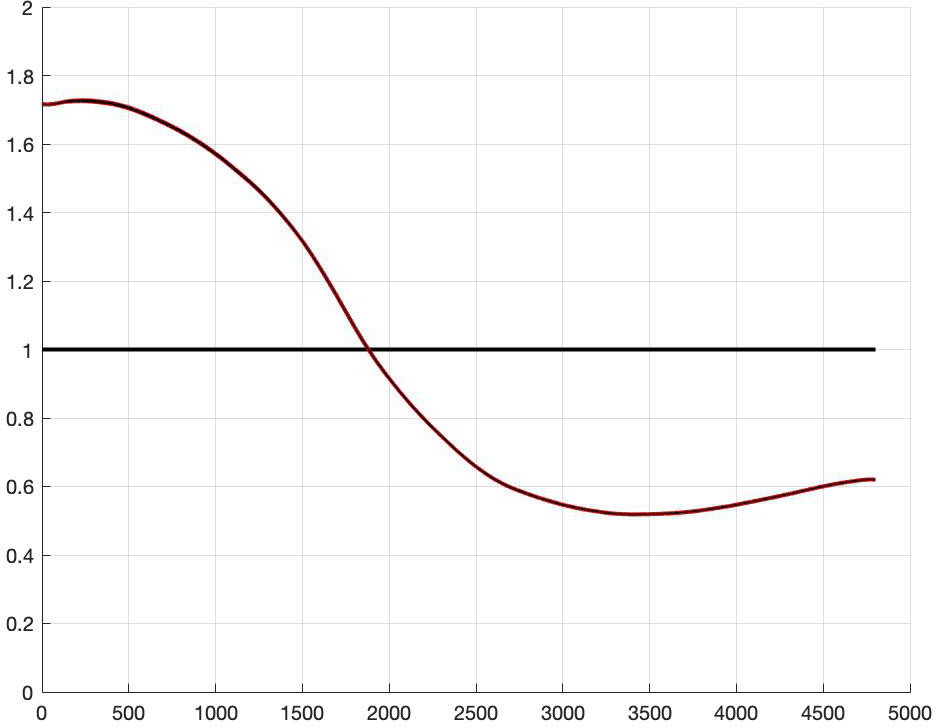} &  &
\includegraphics[height=2.8cm]{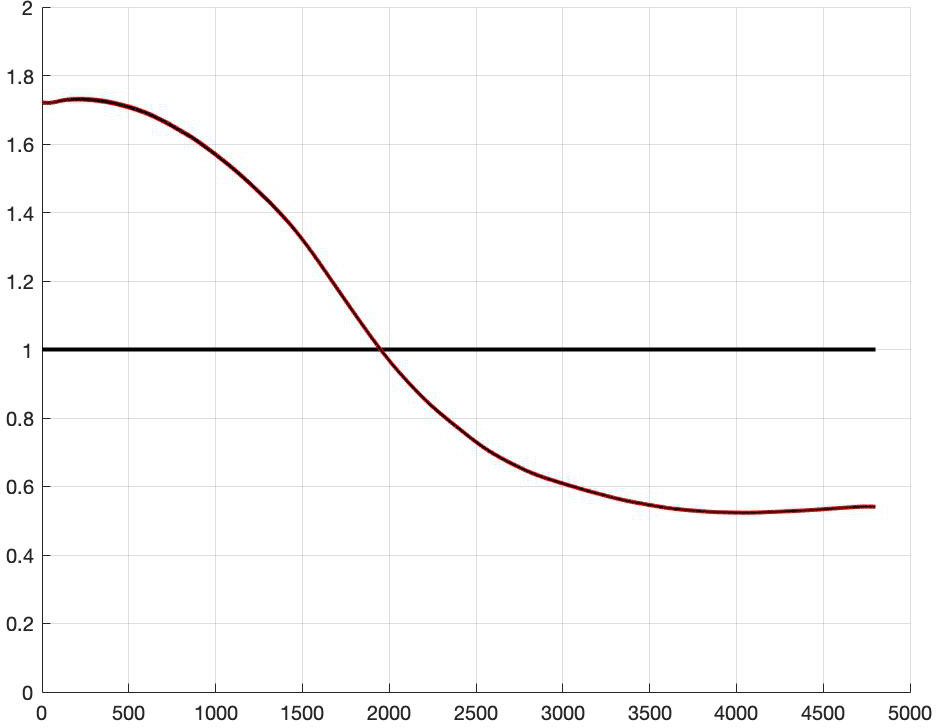} \\ \\

$L=0.8$ & & $L=0.6$ & & $L=0.4$ & & $L=0.2$
\end{tabular}
\caption{Comparison between the optimal social planner policy with limited containment $L$. The figures in the first row show  
the evolution of the (average) containment policy through the value of the optimal control $\xi_t$; the figures in the second row show the (average) evolution of the instantaneous reproduction number $\mathcal{R}_t=\frac{\beta_t}{\alpha}$; the figures in the third row show the evolution of the (average) percentage of susceptible (in blue), infected (in red) and recovered (in green) individuals; the figures in the fourth row show the (average) evolution of the product $\mathcal{R}_t\cdot S_t$.
The limited level of containment varies with the columns: the first column treats the case $L=0.8$, the second column the case $L=0.6$, 
 the third column the case $L=0.4$ and the last column the case $L=0.2$.}
\label{fig2}
\end{center}
\end{figure}

\begin{table}[htb]
\begin{center}
\begin{tabular}{lcccccccccc}
& & $L=0.8$ & & $L=0.6$ & & $L=0.4$ & & $L=0.2$\\
\hline
First day of containment           && {54} && {54} && {54} && {54}\\
%Days of high containment && {50} && {62} && {75} && {83} && {104}\\
Recovered                                 && {52\%} && {58\%} && {61\%} && {68\%}\\
%Recovered increase                 && -- &&   \red{2\%} &&  \red{-4\%} && \red{-4\%} && \red{18\%}\\
%\red{Value functional} &&\\
%\red{Value functional increase} &&\\
\hline
\end{tabular}
\caption{Optimal social planner policy with different values of limited containment $L$.}\label{tab1}
\end{center}
\end{table}

Clearly, the larger $L$ is, the smaller are the social costs (by definition of the value function). 
%From Table \ref{tab1} we notice that the length of the optimal period of containment also decreases with the respect to $L$, meaning that a more moderate lockdown requires a longer duration. 
Our experiment shows that for $L=0.4, 0.6, 0.8, 1$, the final percentage of recovered (hence of the total amount of infected) in average ranges from $52\%$ (case $L=0.8$) up to $68\%$ (case $L=0.2$). In all the cases, the optimal containment starts at the maximal rate and the first day of containment is substantially the same (around day $54$).

Different ceilings $L$ on the containment strategies also affect the values and the fluctuations' size of the reproduction number $\frac{\beta_t}{\alpha}$: smaller values of $L$ correspond to milder variation of the reproduction number $\mathcal{R}_t$ of size $0.3$, whereas larger values of $L$ lead to rapid changes of $\mathcal{R}_t$ which reaches levels smaller than $1$ (less than $0.8$ for $L=0.8$ and less than $0.6$ for $L=1$). In all the cases, $\mathcal{R}_t S_t$ lies strictly below $1$ after a certain date, which is decreasing with respect to $L$ (see the last column in Figures \ref{fig1} and \ref{fig2}). Notice that without any containment policies, $\mathcal{R}_t S_t$ decreases on time due to a natural ``herd-immunity'' effect. On the other, when lockdowns are in place, we observe a faster decrease of $\mathcal{R}_t S_t$ which is forced by the initial vigorous policymaker's actions. The final relaxation of the latter then allows for an increase of $\mathcal{R}_t S_t$, which is however constrained below the critical level of $1$. Such an effect is monotone decreasing with respect to $L$.

%%%%%%%%%%%%%%%%%%%%%%%%%%%%%%%%%%

\subsection{The Role of Uncertainty}
\label{sec:sigma}

{{
The main new feature of our model is to consider a (controlled) stochastic transmission rate in the framework of the classical SIR model. In this section we study numerically how an increase of the fluctuations of the transmission rate affects the optimal solution. In particular, in Figure \ref{fig3} the volatility $\sigma$ takes values $1$, $5$ and $10$, thus leading to fluctuations of $\beta$ of order $10^{-2}$, $5 \times 10^{-2}$, and $10^{-1}$, respectively (indeed, recall that $\sigma(\beta)=\sigma \beta(\gamma-\beta)$ attains its maximum at $\gamma/2$ and $\gamma=0.16$).
\begin{figure}[htb]
\begin{center}
\begin{tabular}{ccccccc}
\includegraphics[height=2.8cm]{xi1} & &
\includegraphics[height=2.8cm]{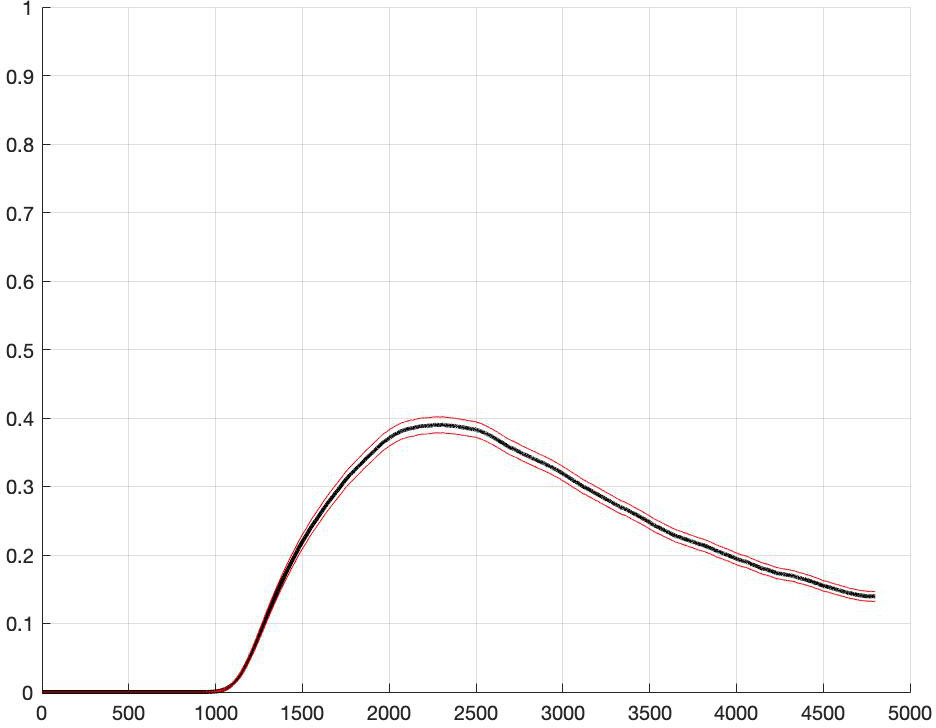} & &
\includegraphics[height=2.8cm]{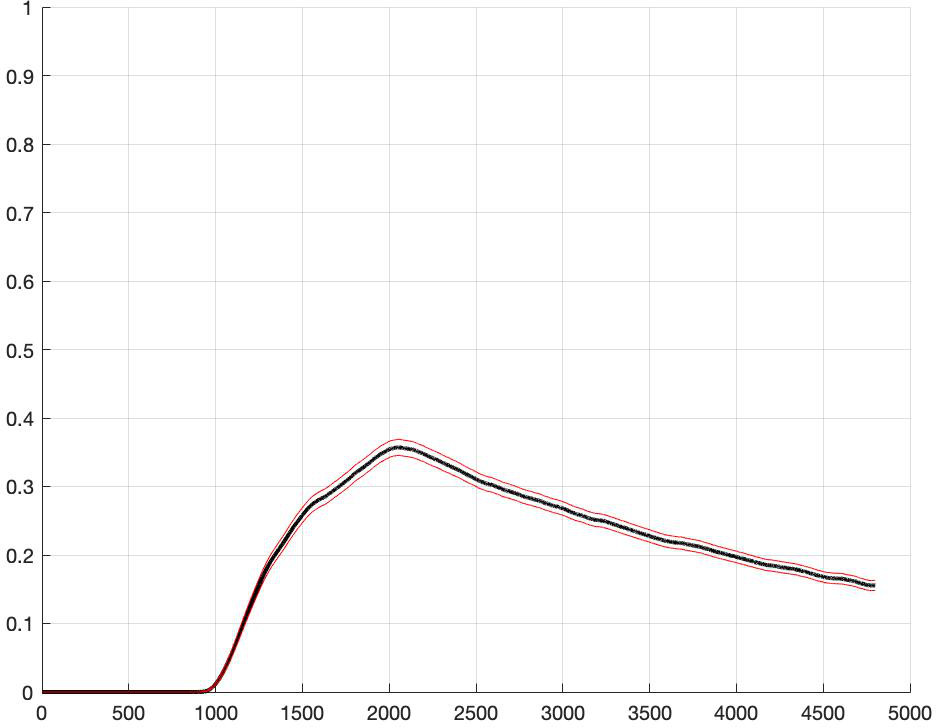} \\ \\
%(a) $L=0.2$ & & (b) $L=0.4$ & & (c) $L=0.6$ & & (d) $L=0.8$\\ \\
%\includegraphics[height=2.8cm]{figures/R1} & &
%\includegraphics[height=2.8cm]{figures/Rsigma5} & &
%\includegraphics[height=2.8cm]{figures/Rsigma10} \\ \\
%(a) $L=0.2$ & & (b) $L=0.4$ & & (c) $L=0.6$ & & (d) $L=0.8$\\ \\
\includegraphics[height=2.8cm]{SIR1} &  &
\includegraphics[height=2.8cm]{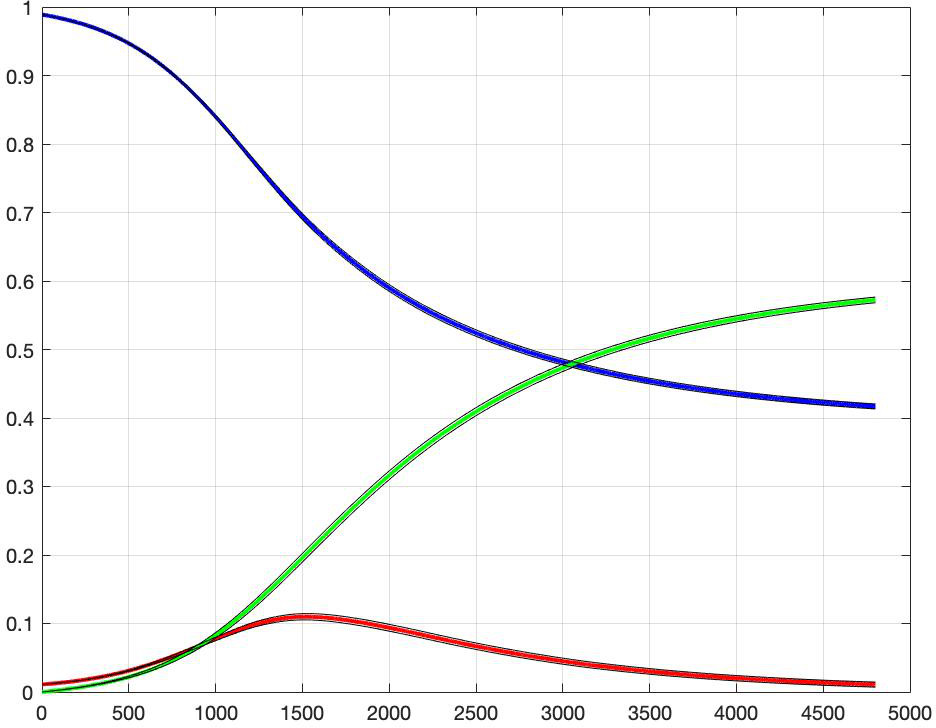} &  &
\includegraphics[height=2.8cm]{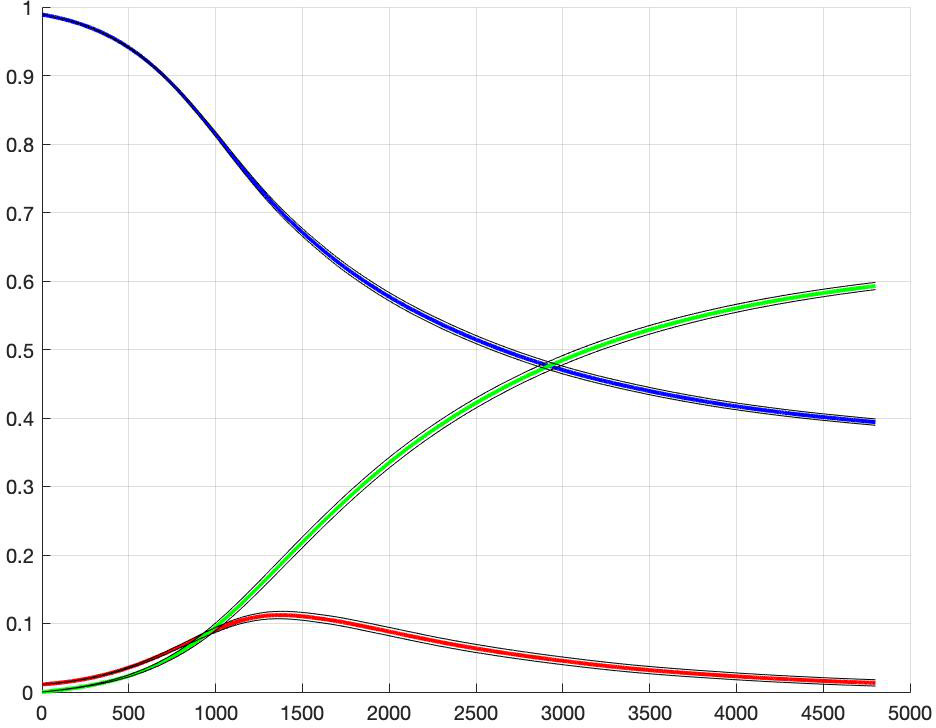} \\ \\

$\sigma=1$ & & $\sigma=5$ & & $\sigma=10$ 
\end{tabular}
\caption{Comparison between the optimal social planner policy with different $\sigma$, when $L=1$. The figures in the first row show  
the evolution of the (average) containment policy through the value of the optimal control $\xi_t$; the figures in the second row show the evolution of the (average) percentage of susceptible (in blue), infected (in red) and recovered (in green) individuals.
The level of $\sigma$ varies with the columns: the first column treats the case $\sigma=1$, the second column the case $\sigma=5$, 
 the third column the case $\sigma=10$.}
\label{fig3}
\end{center}
\end{figure}

We observe from Figure \ref{fig3} that larger fluctuations of $\beta$ have the effect of anticipating the beginning of the lockdown policies, and of diluting the actions over a longer period. Indeed, when $\sigma=5$ and $\sigma=10$, the optimal lockdown policy starts around day $46$ and $42$, respectively, in contrast to day $54$ of the case $\sigma=1$. Moreover, when $\sigma$ increases, the maximal employed lockdown intensity reduces and the level of containment stabilizes at a larger value in the long run. This can be explained by thinking that an increase in the fluctuations of the transmission rate induces the policymaker to act earlier and over a longer period in order to prevent positive large shocks of $\beta$. However, in order to dam the social costs resulting from a longer period of restrictions, the maximal intensity of the lockdown policy should be reduced.

Moreover, such a spread of the optimal lockdown policy gives rise to an increase of the final percentage of recovered (which is circa $58\%$ and $60\%$ when $\sigma =5$ and $\sigma=10$, respectively, and circa $50\%$ when $\sigma =1$). 
}}

%%%%%%%%%%%%%%%%%%%%%%%%%%%%%%%%%%%%%%%%%%%%%%%%%%%%%%%%%%%%%%%%%%%%%%%%%%%%%%%%%%%%%%%%%%%%%%%%%%%%%%%%%%%%%%%%%%%%%%%%%%%%%%%%%%%%%%%%%%%%%%%

\medskip

\section{Conclusions}
\label{sec:concl}

We have studied the problem of a policymaker which during an epidemic is challenged to optimally balance the safeguard of public health and the negative economic impact of severe lockdowns. The policymaker can implement containment policies in order to reduce the trend of the disease's transmission rate, which evolves stochastically in continuous time. In the context of the SIR model, our theoretical analysis allows to identify the minimal social cost function as a classical solution to the corresponding dynamic programming equation, as well as to provide an optimal control in feedback form. 

In a case study in which the transmission rate is a (controlled) mean-reverting diffusion process, numerical experiments show that the optimal lockdown policy is characterized by three distinct phases: the epidemic is first let freely evolve, then vigorously tamed, and finally a less stringent containment should be adopted. Interestingly, in the last period the epidemic's reproduction number is let oscillate strictly above one although the product ``reproduction number $\times$ percentage of susceptible'' is kept strictly below the critical level of one. Hence, under the optimal containment policy, the percentage of infected decreases naturally at an exponential rate and the social planner is then allowed to substantially relax the lockdown in order not to incur too heavy economic costs. {{Moreover, we show that an increase in the fluctuations of the transmission rate gives rise to an earlier beginning of the optimal lockdown policy, which is also diluted over a longer period of time.

We believe that our work is only a first step in enriching the SIR model of a stochastic controlled component and in understanding the policymaker's problem of optimally balancing the safeguard of public health and social wealth. There is still much to be done in order to incorporate other features like the partial detectability of the transmission rate or the role of public investment on the discovery of a vaccination (see Remark \ref{rem:tau} on this). We leave the analysis of the resulting challenging problems for future work.}}

%%%%%%%%%%%%%%%%%%%%%%%%%%%%%%%%%%%%%%%%%%%%%%%%%%%%%%%%%%%%%%%%%%%%%%%%%%%%%%%%%%%%%%%%%%%%%%%%%%%%%%%%%%%%%%%%%%%%%%%%%%%%%%%%%%%%%%%%%%%%%%%

\appendix

\section{Technical results}
\label{sec:app}

\begin{lemma}\label{lemma:app}
Let $\mathcal{O}'$ be an open neighborhood of $\bm 0=(0,0,0)\in\R^3$. Let $W:\mathcal{O}'\to \R$ be a semiconcave function such that $W_z^-(\bm 0)>W_z^+(\bm 0)$.
Then  there exists a sequence of functions  $(\varphi^n)_n\subset C^{2}(\mathcal{O}')$ such that 
\begin{equation}
\label{eq:varphin}
\begin{cases}
\varphi^n(\mathbf{0})=W(\bm 0)=0,\\  \varphi^n\geq W  \ \mbox{in a neighborhood of} \ \bm 0,\\  |D\varphi^n(\bm 0)|\leq L<\infty,\\  \varphi^n_{zz}(\mathbf{0})\stackrel{n\to\infty}{\longrightarrow}-\infty.
\end{cases}
\end{equation}
\end{lemma}

\begin{proof}
Since $W$ is semiconcave, there exists $C_0\geq 0$ such that
$$
\widehat{W}: \mathcal{O}'\to \R, \ \ \ \widehat{W}(x,y,z):= W(x,y,z)-{C}_0\big(x^2+y^2+z^2\big),
$$
is concave. Fix such a $C_0$. 
%
%where $C_0\geq 0$. {{GIO: ma non � questo il $C_0$ di dopo?? confuso...}}
%In the following we denote $\mathbf{q}=(x,y,z)\in\R^3$. Then, we have the following facts 
%\begin{enumerate}[(i)]
%\item  the supergradient $D^+W(\bm 0)$ is a convex closed nonempty set;
%\item  there exist (finite) the left and right derivatives $W_z^-(\bm 0)$ and  $W_z^+(\bm 0)$ of $W$ with respect to $z$;
%\item  $W_z^-(\bm 0) \geq  W_z^+(\bm 0)$.
%\end{enumerate} 
Since $W_z^-(\bm 0)>  W_z^+(\bm 0)$, also $\widehat{W}_z^-(\bm 0)>\widehat{W}_z^+(\bm 0)$ and it is clear that it is equivalent to show the claim for $\widehat{W}$. By \cite[Theorem 23.4]{Rock}, it follows that there exist
$$	\bm{\eta}=(\eta_x,\eta_y,\eta_z), \ \bm{\zeta}=(\zeta_x,\zeta_y,\zeta_z) \in D^+W(\bm 0) \quad \text{such that} \quad \eta_z> \zeta_z.$$
Set
$$g(\bm q):=\langle \bm \eta,\bm q \rangle \wedge \langle \bm \zeta, \bm q\rangle$$ and 
notice that $\widehat{W}(\bm 0)=0=g(\bm 0)$ and that, by concavity,
$$
\widehat{W}(\bm q )\leq g(\bm q) \ \ \forall \bm q\in\mathcal{O}'.
$$
Define
$$A:=\mbox{Span}\{\bm\eta-\bm\zeta\}^\perp,$$ and denote by $\Pi:\R^3\to A$ the orthogonal projection on $A$. Given $\textbf{q}\in \R^3$ we then have the decompositon
$$
\textbf{q}=\Pi \textbf{q}+ \frac{\bm{\eta}-\bm{\zeta}}{|\bm\eta-\bm\zeta|}\,s, \ \ \ s=\frac{\langle \textbf{q},\bm{\eta}-\bm{\zeta}\rangle}{|\bm\eta-\bm\zeta|}.
$$
Define, for $\bm q\in\mathcal{O}'$,
$$
\varphi^n(\bm q):=g(\Pi \bm q)+\psi^n(s),
$$
where 
$$
\psi^n:\R\to\R, \ \ \psi^n(s)=-\frac{n}{2}s^2+\frac{1}{2}\frac{\langle \bm\eta+\bm\zeta, \bm\eta-\bm\zeta\rangle}{|\bm\eta-\bm\zeta|} s.
$$
This sequence realizes \eqref{eq:varphin}.
Indeed, the first two properties 
hold by construction; in particular the second one is due to the fact that
%\begin{itemize}
%\item[(i)]
%the superdifferential of concave functions is convex, so that  $\frac{1}{2}(\bm \eta+\bm\zeta)\in D^+g(\bm 0)$ as $g$ is concave and $\bm \eta,\bm \zeta\in D^+g(\bm 0)$; 
%\item[(ii)] 
we have
$$
g(\bm q)=g(\Pi\bm q)+\begin{cases} 
\frac{\langle \bm \zeta,  \bm\eta-\bm\zeta\rangle}{|\bm\eta-\bm\zeta|}s \ \ \ \mbox{if} \ s\geq 0, \\\\
  \frac{\langle \bm \eta,  \bm\eta-\bm\zeta\rangle}{|\bm\eta-\bm\zeta|}s \ \ \ \mbox{if} \ s<0.
\end{cases}
$$
%\end{itemize}
As for the last two properties, we notice that
$$
D \varphi^n(\bm q)= \Pi\bm \eta \ (=\Pi\bm \zeta)+ \frac{\bm \eta -\bm \zeta}{|\bm \eta-\bm \zeta|}\frac{\d\psi^n}{\d s}(s), 
$$
so
$$
\varphi^n_z(\bm q)= \langle \Pi\bm \eta, (0,0,1)\rangle + \left\langle \frac{\bm \eta -\bm \zeta}{|\bm \eta-\bm \zeta|},(0,0,1)\right\rangle\frac{\d\psi^n}{\d s}(s) = \langle \Pi\bm \eta, (0,0,1)\rangle +  \frac{\eta_z -\zeta_z}{|\bm \eta-\bm \zeta|}\frac{\d\psi^n}{\d s}(s), 
$$
$$
\varphi^n_{zz}(\bm q)=  \frac{\eta_z -\zeta_z}{|\bm \eta-\bm \zeta|}\frac{\d^2\psi^n}{\d s^2}(s), 
$$
which then imply them. 
\end{proof}

Denote by $\textbf{q}=(q_1,q_2,q_3):=(x,y,z)$ an arbitrary point of $\mathcal{O}$. For any multi-index $\alpha:=(\alpha_1,\alpha_2,\alpha_3) \in \mathbb{N}^3$ we denote by $|\alpha| = \sum_{i=1}^3 \alpha_i$ and $D^{\alpha}_{\textbf{q}}=\partial^{|\alpha|}/\partial_{q_1}^{\alpha_1}\dots \partial_{q_3}^{\alpha_3}$, with the convention that $\partial^{0}$ is the identity.

\begin{proposition}
\label{prop:A1}
For any $\textbf{q} \in \mathcal{O}$ the (uncontrolled) process $(\textbf{Q}^{\textbf{q}}_t)_t:=(S^{x,y,z}_t, I^{x,y,z}_t, \beta^z_t)$ admits a transition density $p$ which is absolutely continuous with respect to the Lebesgue measure in $\R^3$, infinitely many times differentiable, and satisfies the Gaussian estimates
\begin{equation}
\label{eq:Gauss1}
p(t, \textbf{q}; \textbf{q}') \leq \frac{C_0(t)(1 + |\textbf{q}|)^{m_0}}{t^{\frac{n_0}{2}}} e^{- \frac{D_0(t)|\textbf{q}'- \textbf{q}|^2}{t}}, \quad \forall  t>0,\, \textbf{q}'=(x',y',z') \in \mathcal{O},
\end{equation}
\begin{equation}
\label{eq:Gauss2}
|D^{\alpha}_{\textbf{q}}p(t, \textbf{q}; \textbf{q}')| \leq \frac{C_{\alpha}(t)(1 + |\textbf{q}|)^{m_{\alpha}}}{t^{\frac{n_{\alpha}}{2}}} e^{- \frac{D_{\alpha}(t)|\textbf{q}'- \textbf{q}|^2}{t}}, \quad \forall  t>0,\, \textbf{q}'=(x',y',z') \in \mathcal{O}.
\end{equation}
%and 
%\begin{equation}
%\label{eq:Gauss3}
%|D^{\alpha}_{\textbf{q}'}p(t, \textbf{q}; \textbf{q}')| \leq \frac{C_{\alpha}(t)(1 + |\textbf{q}|)^{m_{\alpha}}}{t^{\frac{n_{\alpha}}{2}}} e^{- \frac{D_{\alpha}(t)|\textbf{q}'- \textbf{q}|^2}{t}}, \quad \forall  t>0,\, \textbf{q}'=(x',y',z') \in \mathcal{O}.
%\end{equation}
Here, $C_0$, $D_0$, $C_{\alpha}$, and $D_{\alpha}$ are increasing functions of time. 
\end{proposition}

\begin{proof}
Given $f,g \in C^1(\R^3;\R^3)$, define the Lie bracket 
$$[f,g]:=\sum_{j=1}^3 \Big(\frac{\partial g}{\partial q_j} f - \frac{\partial f}{\partial q_j} g\Big).$$
Then, for any given and fixed $\textbf{q}\in \mathcal{O}$, we set
\begin{eqnarray*}
\mu(\textbf{q}):= \left(\begin{array}{c}
-xyz\\xyz-\alpha y\\ b(z,0)
\end{array}\right)
\quad \textrm{ and } \quad 
\Sigma(\textbf{q}):= \left(\begin{array}{c}
0 \quad 0 \quad  0 \\ 0 \quad  0 \quad 0 \\ 0 \quad  0 \quad \sigma(z)
\end{array}\right)
\end{eqnarray*}
and denoting by $\Sigma_i$, $i=1,2,3$, the columns of the matrix $\Sigma$, we construct recursively the set of functions $L_0:=\{\Sigma_1, \Sigma_2, \Sigma_3\}$, $L_{k+1}:= \{[\mu,\varphi], [\Sigma_1,\varphi], [\Sigma_2,\varphi], [\Sigma_3,\varphi]\,:\, \varphi \in L_k\}$, $k\geq0$. We also define $L_{\infty}:=\cup_{k\geq0}L_k$.
We say that the H\"ormander condition holds true at $\textbf{q}\in \mathcal{O}$ if
\begin{equation}
\label{eq:Hormander}
\Span\big\{\varphi(\textbf{q}),\, \varphi \in L_{\infty}\big\} = \R^3.
\end{equation}
Direct calculations show that
\begin{eqnarray*}
L_0(\textbf{q})=\{\Sigma_3\}(\textbf{q})=\left\{\left(\begin{array}{c}
0\\0\\ \sigma(z) 
\end{array}\right)\right\}
\end{eqnarray*}

\begin{eqnarray*}
L_1(\textbf{q}) = \big\{[\mu,\Sigma_3]\big\}(\textbf{q}) = \left\{
\left(\begin{array}{c}
xy\sigma(z)\\ -xy\sigma(z) \\  \sigma_z(z)b(z,0)-\sigma(z) b_z(z,0)
\end{array}\right)\right\}
\end{eqnarray*}
and
\item \begin{eqnarray*}
&& L_2(\textbf{q})= \big\{ [\mu,[\mu,\Sigma_3]], [\Sigma_3,[\mu,\Sigma_3]]\big\}(\textbf{q})\\
&& = \left\{\left(\begin{array}{c}
%-sxy(a - z)(dz - 2ap + 3pz)\\
%-psxy(2a^2 - 5az + 3z^2)\\
%-p^2s(a - z)^2
xy(2b(z,0) \sigma_z(z) - \alpha \sigma(z) - \sigma(z) b_z(z,0))\\ 
xy (\sigma(z)b_z(z,0)  - 2 b(z,0) \sigma_z(z))\\         
b(z,0)^2 \sigma_{zz}(z) - \sigma(z) b(z,0) b_{zz}(z,0) + \sigma(z) b_z(z,0)^2 - \sigma_z(z) b(z,0) b_z(z,0)
 \end{array}\right),\right.\\
&&\left.\left(\begin{array}{c}
xy \sigma(z)\sigma_z(z)\\                           
-xy \sigma(z) \sigma_z(z)\\ 
b(z,0) \sigma(z) \sigma_{zz}(z) - \sigma(z)^2 b_{zz}(z,0) - b(z,0) \sigma_z(z)^2 + \sigma(z) b_z(z,0) \sigma_z(z)
%s^2xyz(a^2 - 3az + 2z^2)\\
%-s^2xyz(a^2 - 3az + 2z^2)\\
% -aps^2(a - z)^2
 \end{array}\right)
\right\}
\end{eqnarray*}
Hence, the matrix associated to $(L_0 \cup L_1 \cup L_2)(\textbf{q})$ has the sub-matrix formed by all its rows and its first three columns with determinant $-\alpha x^2 y^2 \sigma^3(z) < 0$. Hence, \eqref{eq:Hormander} holds true on $\mathcal{O}$ given the arbitrariness of $\textbf{q}$.

Therefore, by Theorem 2.3.3 in \cite{Nualart}, for any $t>0$ the uncontrolled process $(S^{x,y,z}_t, I^{x,y,z}_t, \beta^z_t)_t$ admits a transition density $p$ which is absolutely continuous with respect to the Lebesgue measure in $\R^3$, and infinitely many times differentiable. Moreover, Theorem 9 and Remark 11 in \cite{Bally} (see also \cite{KuStr}) show that $p$ satisfies the Gaussian estimates \eqref{eq:Gauss1} and \eqref{eq:Gauss2}. This completes the proof.
\end{proof}

{{
\begin{proposition}
\label{prop:betadyn}
The dynamics of $\beta$ as in \eqref{eq:ZOUbis} satisfy Assumption \ref{assZ}.
\end{proposition}
\begin{proof}
For any $z \in \R$ and $\xi \in [0,L]$, recall that $b(z,\xi)=\vartheta(\widehat{\beta}(L-\xi) - z)$ and define
%\begin{equation*}
%\widetilde{b}(z,\xi)=\left\{
%\begin{array}{lr}
%\displaystyle \vartheta\widehat{\beta}(1-\xi)\,\,\qquad\,\,\, \mbox{if $z \in (-\infty, 0]$}\\[+14pt]
%\displaystyle \vartheta(\widehat{\beta}(1-\xi) - z) \,\,\quad \mbox{if $z \in (0,\gamma)$}\\[+14pt]
%\displaystyle \vartheta(\widehat{\beta}(1-\xi) - \gamma)\,\,\quad \mbox{if $z \in [\gamma,\infty)$}
%\end{array}
%\right.
%\end{equation*}
%and
\begin{equation}
\label{sigmatilde}
\widetilde{\sigma}(z)=\left\{
\begin{array}{lr}
\displaystyle 0\,\,\qquad \qquad \ \  \mbox{if $z \in (-\infty, 0]$},\\[+14pt]
\displaystyle \sigma z (\gamma -z)\,\,\quad \mbox{if $z \in (0,\gamma)$},\\[+14pt]
\displaystyle 0\,\,\qquad \qquad \ \ \mbox{if $z \in [\gamma,\infty)$}.
\end{array}
\right.
\end{equation}
Then, for $(\xi_t)_t \in \mathcal{A}$, introduce the stochastic differential equation 
\begin{equation}
\label{eq:betatilde}
\d \widetilde{\beta}_t = b(\widetilde{\beta}_t, \xi_t) \d t + \widetilde{\sigma}(\widetilde{\beta}_t)\d W_t, \quad \widetilde{\beta}_0=z \in \R.
\end{equation}
Because $b$ and $\widetilde{\sigma}$ are Lipschitz-continuous and have sublinear growth, uniformly with respect to $\xi$, for any $(\xi_t)_t \in \mathcal{A}$ there exists a unique strong solution to \eqref{eq:betatilde} starting at $z \in \R$ (see, e.g., Theorem 7 in Chapter V of \cite{Protter}). We denote such a solution by $(\widetilde{\beta}^{\xi,z}_t)_t$.
Since $\xi \mapsto b(z,\xi)$ is decreasing, by Theorem 54 in Chapter V of \cite{Protter}, we have
\begin{equation}\label{comparison}
 \widetilde{\beta}^{L,z}_t\leq \widetilde{\beta}^{\xi,z}_t  \leq \widetilde{\beta}^{0,z}_t, \ \ \ \ \forall t\geq 0 \ \mbox{a.s.},
 \end{equation} where $(\widetilde{\beta}^{L,z}_t)_t$ and $(\widetilde{\beta}^{0,z}_t)_t$) are, respectively, the solution to \eqref{eq:betatilde} with $\xi_t \equiv L$ and with $\xi_t \equiv 0$. On the other hand, by  Feller's test for explosion (cf.\ Proposition 5.22 in Chapter 5.5 of \cite{KS}), it can be checked $\widetilde{\beta}^{L,z}_t >0$ for all $t\geq 0$ a.s.\ and $\widetilde{\beta}^{0,z}_t < \gamma$ for all $t\geq 0$ a.s. Hence, by \eqref{comparison}, we get $\widetilde{\beta}^{\xi,z}_t \in (0,\gamma)$ for all $t\geq 0$ a.s.\ This proves that the SDE \eqref{eq:ZOUbis} admits a unique strong solution which lies within  the open interval  $\mathcal{I}=(0,\gamma)$ (cf.\ Assumption \ref{assZ}-(i)). 

Given the boundedness of the interval $\mathcal{I}=(0,\gamma)$, using the expression of $b$ and \eqref{sigmatilde} it is straightforward to verify that (ii) and (iii) of Assumption \ref{assZ} are satisfied as well.
\end{proof}
}}

%%%%%%%%%%%%%%%%%%%%%%%%%%%%%%%%%%%%%%%%%%%%%%%%%%%%%%%%%%%%%%%%%%%%%%%%%%%%%%%%%%%%%%%%%%%%%%%%%%%%%%%%%%%%%%%%%%%%%%%%%%%%%%%%%%%%%%%%%%%%%%%

\medskip

\indent \textbf{Acknowledgments.} The authors thank Frank Riedel, Mauro Rosestolato, three anonymous Referees and the Guest Editors for interesting comments and suggestions.

%%%%%%%%%%%%%%%%%%%%%%%%%%%%%%%%%%%%%%%%%%%%%%%%%%%%%%%%%%%%%%%%%%%%%%%%%%%%%%%%%%%%%%%%%%%%%%


\begin{thebibliography}{199}

\bibitem{Acemoglu} \textsc{Acemoglu, D., Chernozhukov, V., Werning, I., Whinston, M.D.}\ (2020). Optimal Targeted Lockdowns in a Multi-Group SIR Model. Working Paper 27102, National Bureau of Economic Research. Available online at \url{http://www.nber.org/papers/w27102}.

\bibitem{Allen2017} \textsc{Allen, L.J.S.}\ (2017). A Primer on Stochastic Epidemic Models: Formulation, Numerical Simulation, and Analysis. \emph{Infect.\ Dis.\ Model.}\ \textbf{2}, 128--142.

\bibitem{AlvarezL} \textsc{Alvarez, L.H.R.}\ (2001). Singular Stochastic Control, Linear Diffusions, and Optimal Stopping: A Class of Solvable Problems. \emph{SIAM J.\  Control Optim.}\ \textbf{39(6)}, pp.\ 1697--1710. 

\bibitem{Alvarez} \textsc{Alvarez, F.E., Argente, D., Lippi, F.}\ (2020). A Simple Planning Problem for COVID-19 Lockdown, Testing and Tracing. Forthcoming on \emph{Am.\ Econ.\ Rev.: Insights}. Preprint vailable online at \url{https://www.nber.org/papers/w26981}

%\bibitem{Anderson} \textsc{Andersson, H., Britton, T.}\ (2012). Stochastic Epidemic Models and their Statistical Analysis (Vol.\ 151). Springer Science and Business Media.

\bibitem{Asprietal} \textsc{Aspri, A., Beretta, E., Gandolfi, A., Wasmer, E.}\ (2020). Mortality Containment vs.\ Economic Opening: Optimal Policies in a SEIARD Model. Preprint on \textbf{arXiv}: 2006.00085. Available online at \url{https://arxiv.org/abs/2006.00085}.

\bibitem{Bally} \textsc{Bally, V.}\ (2003). \emph{An Elementary Introduction to Malliavin calculus}. INRIA Rapport de Recerche n.\ 4718. Available online at \url{http://www.crm.cat/en/Activities/Documents/RR-4718.pdf}

\bibitem{Erhan} \textsc{Bayraktar, E., Cohen, A., Nellis, A.}\ (2020). A Macroeconomic SIR Model for Covid-19. Preprint on \textbf{arXiv}: 2006.16389. Available online at \url{https://arxiv.org/abs/2006.16389}.

\bibitem{Behncke} \textsc{Behncke, H.}\ (2000). \emph{Optimal Control of Deterministic Epidemics}. Optimal Control Applications and Methods \textbf{21(6)}, pp.\ 269--285.

\bibitem{Belak} \textsc{Belak, C., Christensen, C., Seifried, F.T.}\ (2017). A General Verification Result for Stochastic Impulse Control Problems. \emph{SIAM J.\  Control Optim.}\ \textbf{55(2)}, pp.\ 627--649. 

\bibitem{CS} \textsc{Cannarsa, P., Sinestrari, C.}\ (2014). \emph{Semiconcave Functions, Hamilton-Jacobi Equations, and Optimal Control}. Progress in Nonlinear Differential Equations and their Applications, Volume 58. Birkh\"auser.

\bibitem{CIL} \textsc{Crandall M., Ishii H., Lions P.L.}\ (1992). User's Guide to Viscosity Solutions of Second Order Partial Differential Equations. \emph{Bull.\ Amer.\ Math.\ Soc.}\ \textbf{27(1)}, pp.\ 1--67. 

\bibitem{Donsimoni} \textsc{Donsimoni, J.R., Glawion, R., Placher, B., W\"alde, K.}\ (2020). Projecting the Spread of COVID-19 for Germany. \emph{Ger.\ Econ.\ Rev.}\ \textbf{21(2)}, DOI: \url{https://doi.org/10.1515/ger-2020-0031}

\bibitem{Favero} \textsc{Favero, C., Ichino, A., Rustichini, A.}\ (2020). Restating the Economy while Saving Lives under Covid-19. Preprint at SSRN: \url{https://papers.ssrn.com/sol3/papers.cfm?abstract_id=3580626}

\bibitem{FedericoTacconi} \textsc{Federico, S., Tacconi, E.}\ (2014). Dynamic Programming for Optimal Control Problems with Delays in the Control Variable. \emph{SIAM J.\ Control Optim.} \textbf{52(2)}, pp.\ 1203--1236.

\bibitem{Ferrari} \textsc{Ferrari, G.}\ (2018). On the Optimal Management of Public Debt: A Singular Stochastic Control Problem. \emph{SIAM J.\ Control Optim.} \textbf{56(3)}, pp.\ 2036--2073.


\bibitem{Gollier} \textsc{Gollier, C.}\ (2020). Cost-Benefit Analysis of Age-Specific Deconfinement Strategies. \emph{Covid Economics} 24, pp.\ 1--31. Available online at: \url{https://www.tse-fr.eu/sites/default/files/TSE/documents/doc/by/gollier/covideconomics_dynamic.pdf}

\bibitem{Greenwood-Gordillo} \textsc{Greenwood, P.E., Gordillo L.F.}\ (2009). \emph{Stochastic Epidemic Modeling}. In: Chowell G., Hyman J.M., Bettencourt L.M.A., Castillo-Chavez C.\ (eds) Mathematical and Statistical Estimation Approaches in Epidemiology. Springer, Dordrecht. 

\bibitem{Hotz} \textsc{Hotz, T., Glock, M., Heyder, S., Semper, S., B\"ohle, A., Kr\"amer, A.}\ (2020). Monitoring the Spread of Covid-19 by Estimating Reproduction Numbers over Time. Preprint available at: \url{https://stochastik-tu-ilmenau.github.io/COVID-19/reports/repronum/repronum.pdf}

\bibitem{Jacco} \textsc{Huberts, N., Thijssen, J.}\ (2020). Optimal Timing of Interventions during an Epidemic. Preprint available at SSRN: \url{https://papers.ssrn.com/sol3/papers.cfm?abstract_id=3607048}

\bibitem{Jiang} \textsc{Jiang, D., Yu, J.}\ (2011). Asymptotic Behavior of Global Positive Solution to a Stochastic SIR Model. \emph{Math.\ Comput.\ Modelling} \textbf{54}, pp.\ 221--232.

\bibitem{Kantner} \textsc{Kantner, M.}\ (2020). Beyond just ``Flattening the Curve'': Optimal Control of Epidemics with Purely Nonpharmaceutical Interventions. Preprint on \textbf{arXiv}: 2004.09471. Available online at \url{https://arxiv.org/abs/2004.09471}

\bibitem{KS} \textsc{Karatzas, I., Shreve, S.}\ (1991). \emph{Brownian Motion and Stochastic Calculus. Second Edition}. Springer-Verlag, New York.

\bibitem{KermackMcKendrick} \textsc{Kermack, W.O., McKendrick, A.G.}\ (1927). Contributions to the Mathematical Theory of Epidemics, Part I. \emph{Proc.\ R.\ Soc.\ Edinb.\ A} \textbf{115}, pp.\ 700--721.

\bibitem{KuStr} \textsc{Kusuoka, S., Stroock, D.}\ (1985). Applications of the Malliavin Calculus, Part II. \emph{J.\ Fac.\ Sci.\ Univ.\ Tokyo} \textbf{32}, pp.\ 1--76.

\bibitem{KruseStrack} \textsc{Kruse, T., Strack, P.}\ (2020). Optimal Control of an Epidemic through Social Distancing. Preprint at SSRN: \url{https://ssrn.com/abstract=3581295 or http://dx.doi.org/10.2139/ssrn.3581295}.

%\bibitem{Krylov} \textsc{Krylov, N.V.}\ (1980). \emph{Controlled Diffusion Processes}. Springer-Verlag.

\bibitem{Micloetal} \textsc{Miclo, L., Spiroz, D., Weibull, J.}\ (2020). Optimal Epidemic Suppression under an ICU Constraint. Preprint on \textbf{arXiv}: 2005.01327. Available online at \url{https://arxiv.org/abs/2005.01327}

\bibitem{Nualart} \textsc{Nualart, D.}\ (2006). \emph{The Malliavin Calculus and Related Topics}. Second edition. Springer-Verlag.

\bibitem{Pham} \textsc{Pham, H.}\ (2009). \emph{Continuous-time Stochastic Control and Optimization with Financial Applications}. Springer-Verlag.

\bibitem{Protter} \textsc{Protter, P.E.}\ (2004). \emph{Stochastic Integration and Differential Equations}. Springer.

\bibitem{Rock} \textsc{Rockafellar T.}\ (1970). \emph{Convex Analysis}. Princeton University Press.

\bibitem{Toda} \textsc{Toda, A.A.}\ (2020). \emph{Susceptible-Infected-Recovered (SIR) Dynamics of Covid-19 and Economic Impact}. Preprint on \textbf{arXiv}: 2003.11221.

\bibitem{Tornatore} \textsc{Tornatore, E., Buccellato, S.M., Vetro, P.}\ (2005). Stability of a Stochastic SIR System. \emph{Physica A} \textbf{354} pp.\ 111--126.

\bibitem{V} \textsc{Veretennikov, A.Y.}\ (1979). On the Strong Solutions of Stochastic Differential Equations. \emph{Theory Probab.\ Appl.}\ 24(2), pp.\ 354--366.

\bibitem{YongZhou1999} \textsc{Yong, J., Zhou, X.Y.}\ (1999). \emph{Stochastic Control - Hamiltonian Systems and HJB Equations}. Springer.

\end{thebibliography}
\end{document}